\documentclass[11pt]{article}

\usepackage[margin=1.2in]{geometry}

\usepackage{amsmath}
\usepackage{amsthm}
\usepackage{graphicx}
\usepackage{color}
\usepackage{amsfonts}
\usepackage{amssymb}
\usepackage{amscd}
\usepackage{url}

\newcommand{\bdes}{\begin{description}}
\newcommand{\edes}{\end{description}}

\newcommand{\bal}{\begin{align}}
\newcommand{\eal}{\end{align}}

\newcommand{\bnum}{\begin{enumerate}}
\newcommand{\enum}{\end{enumerate}}

\newcommand{\bit}{\begin{itemize}}
\newcommand{\eit}{\end{itemize}}

\newcommand{\bea}{\begin{eqnarray}}
\newcommand{\eea}{\end{eqnarray}}
\newcommand{\be}{\begin{equation}}
\newcommand{\ee}{\end{equation}}

\newcommand{\baray}{\begin{array}}
\newcommand{\earay}{\end{array}}

\newcommand{\bsry}{\begin{subarray}}
\newcommand{\esry}{\end{subarray}}

\newcommand{\bca}{\begin{cases}}
\newcommand{\eca}{\end{cases}}

\newcommand{\bcen}{\begin{center}}
\newcommand{\ecen}{\end{center}}

\newcommand{\bbm}{\begin{bmatrix}}
\newcommand{\ebm}{\end{bmatrix}}

\newcommand{\bmx}{\begin{matrix}}
\newcommand{\emx}{\end{matrix}}

\newcommand{\bpm}{\begin{pmatrix}}
\newcommand{\epm}{\end{pmatrix}}

\newcommand{\btab}{\begin{tabular}}
\newcommand{\etab}{\end{tabular}}

\theoremstyle{plain}

\newtheorem{theorem}{Theorem}[section]

\newtheorem{prop}[theorem]{Proposition}

\newtheorem{lemma}[theorem]{Lemma}
\newtheorem{cor}[theorem]{Corollary}
\newtheorem{corollary}[theorem]{Corollary}

\theoremstyle{definition}
\newtheorem{example}[theorem]{Example}

\newtheorem{remark}[theorem]{Remark}

\renewcommand{\subsection}[1]{
    \stepcounter{subsection}
    \settowidth{\hangindent}{\bf\thesubsection.~}
    \hangafter=1
    \bigskip\bigskip\noindent
    {\bf\hbox{\thesubsection.~}#1}\par
    \nobreak
    \medskip
}

\begin{document}

\title{Recursively determined representing measures \\ 
for bivariate truncated moment sequences}
\author{Ra\' ul E. Curto
\footnote{Research was partially supported by NSF Grant DMS-0801168.} \\
and \\
Lawrence Fialkow
\footnote{Research was partially supported by NSF Grant DMS-0758378.}}

\maketitle

\begin{abstract}

A theorem of Bayer and Teichmann [BT] implies that if a finite
real multisequence $\beta\equiv \beta^{(2d)}$ has a representing
measure, then the associated moment matrix $M_{d}$ admits
positive, recursively generated moment matrix extensions $M_{d+1},~M_{d+2},\ldots$. \
For a bivariate recursively determinate $M_{d}$, we show that the existence of
positive, recursively generated extensions $M_{d+1},\ldots,M_{2d-1}$ is sufficient
for a measure. \ Examples illustrate that all of these extensions may be required to 
show that $\beta$ has a measure. \ We describe in detail a constructive procedure for
determining whether such extensions exist. \ Under mild additional
hypotheses, we show that $M_{d}$ admits an extension $M_{d+1}$
which has many of the properties of a positive,
recursively generated extension.

\end{abstract}




\section{Introduction}\label{Intro}
\setcounter{equation}{0}
Let $\beta\equiv \beta^{(2d)}:= \{\beta_{ij}\}_{i,j\ge 0, i+j\le 2d}$
denote a real bivariate moment sequence of degree $2d$. \ The Truncated
Moment Problem seeks conditions on $\beta$ for the existence of a
positive Borel measure $\mu$ on $\mathbb{R}^{2}$ such that
\begin{equation} \label{beta}
\beta_{ij} = \int_{\mathbb{R}^{2}} x^{i}y^{j}d\mu \quad(i,j\ge 0,~~i+j\le 2d).
\end{equation}
A result of \cite{tcmp10} shows that $\beta$ admits a finitely atomic {\it{representing measure}} $\mu$
(as in (\ref{beta})) if and only if $M_{d}\equiv M_{d}(\beta)$, the {\it{moment matrix}}
associated with $\beta$, admits a {\it flat extension} $M_{d+k+1}$, i.e., an extension to a positive semidefinite moment
matrix $M_{d+k+1}$ such that $\text{rank}~M_{d+k+1} = 
\text{rank}~M_{d+k}$. \ The extension of this result to general representing measures follows from a theorem
of C. Bayer and J. Teichmann \cite{BT}, which implies that if $\beta$ has a representing measure, then it has a finitely atomic 
representing measure (cf. \cite[Section 2]{finitevariety}, \cite[Section 1]{tcmp11}). \ At present, for a general moment matrix, there
is no known concrete test for the existence of a flat extension $M_{d+k+1}$.
In this note, for the class of bivariate {\it{recursively determinate}}
moment matrices, we present a detailed analysis of an algorithm of \cite{finitevariety}
that can be used in numerical examples to determine the existence or 
nonexistence of flat extensions (and representing measures). \ This algorithm determines the existence or 
nonexistence of positive, recursively generated extensions $M_{d+1},\ldots,M_{2d-1}$, at least one of which
must be a flat extension in the case when there is a measure. \ Theorem \ref{gridthm} shows that there are
sequences $\beta^{(2d)}$ for which the first flat extension occurs at $M_{2d-1}$, so all of the above extensions must be
computed in order to recognize that there is a measure. \ This result stands in sharp contrast to traditional truncated moment
theorems (concerning representing measures supported in $\mathbb{R}$, $[a,b]$, $[0,+\infty)$, or in a planar curve of degree $2$), 
which express the existence of a measure in terms of tests closely related to the original moment data (cf.
Remark \ref{newrmk} below and \cite{Houston}, \cite{tcmp1}, \cite{tcmp3}, \cite{tcmp11}, \cite{finitevariety}). \ Here we see that, at least within 
the framework of moment matrix extensions, we may need to go far from the original data to resolve the existence of a 
measure. \ In Theorems \ref{rdext} and \ref{RDnew} we show that 
under mild additional hypotheses on $M_d$, the implementation of each extension step, from $M_{d+j}$ to $M_{d+j+1}$,
leading to a flat extension $M_{d+k+1}$, consists of simply verifying a matrix positivity condition.

Let $\mathcal{P}_{d}\equiv \mathbb{R}[x,y]_{d}$ denote the bivariate real
polynomials of degree at most $d$. \ For $p\in \mathcal{P}_{d}$, $p(x,y) \equiv
\displaystyle \sum \limits_{i,j\ge 0, i+j\le d} a_{ij}x^{i}y^{j}$, let $\hat{p}:=(a_{ij})$ denote
the vector of coefficients with respect to the basis for $\mathcal{P}_{d}$
consisting of monomials in degree-lexicographic order, i.e.,
$1,x,y,x^{2},xy,y^{2},\ldots,x^{d},\newline
\ldots,y^{d}$. \ Let 
$L_{\beta}:\mathcal{P}_{2d}\longrightarrow \mathbb{R}$ denote the 
Riesz functional, defined by $L_{\beta}(\displaystyle \sum \limits_{i,j\ge 0,i+j\le 2d} a_{ij}x^{i}y^{j})
:=\sum a_{ij} \beta_{ij}$. \ The moment matrix $M_{d}$, whose rows and
columns are indexed by the monomials in $\mathcal{P}_{d}$, is defined by
$\langle M_{d}\hat{p},\hat{q} \rangle := L_{\beta}(pq)$ ($p,q\in \mathcal{P}_{d}$).
We denote the successive rows and columns of $M_{d}$ by $1,~X,~Y,\ldots,
X^{d},\ldots,Y^{d}$; thus, the entry in row $X^iY^j$, column $X^kY^{\ell}$, which we denote by
$\langle X^{k}Y^{\ell},X^{i}Y^{j} \rangle$, is equal to $\beta_{i+k,j+\ell}$. \ We may denote a linear combination of
 rows or columns  by $p(X,Y):= \sum a_{ij}X^{i}Y^{j}$ for some
$p \equiv \sum a_{ij}x^{i}y^{j}
\in \mathcal{P}_{d}$; note that $p(X,Y) = M_{d}\hat{p}$. \ We say that $M_{d}$ is {\it{recursively
 generated}} if $ker~M_{d}$ has the following ideal-like property:
 \begin{equation} \label{recgen}
p,q,pq\in \mathcal{P}_{d},~p(X,Y)=0 \Longrightarrow (pq)(X,Y)=0.
\end{equation}

If $\beta$ has a representing measure, then $M_{d}$ is positive semidefinite
and recursively generated \cite{tcmp10} (and in one variable these conditions are 
sufficient for the existence of a representing measure \cite{Houston}). \ Moreover, from \cite{BT}, $\beta$ actually admits 
a {\it finitely atomic} representing measure $\mu$, which therefore has finite moments of all orders; it follows that $M_{d}$ admits positive,
recursively generated moment matrix extensions of all orders, namely $M_{d+1}[\mu],\ldots,M_{d+k}[\mu],\ldots$. \ Let us consider a moment matrix
extension 
$$M_{d+1}\equiv 
\bpm
M_{d} & B(d+1) \\
B(d+1)^{T} & C(d+1)   
\epm,$$
where the block $B(d+1)$ includes new moments of degree $2d+1$ (as well as old moments of degrees $d+1,\ldots,2d$), 
and block $C(d+1)$ consists of new moments of degree $2d+2$. \ We denote the columns of $B(d+1)$ by $X^{d+1},\ldots,Y^{d+1}$, and we say that 
$(M_d \; B(d+1))$ is {\it recursively generated} if (\ref{recgen}) holds in its column space, but with $p,q,pq\in \mathcal{P}_{d+1}$. \
 $M_{d+1}$ is positive semidefinite if and only if (i) $M_{d}$ is positive semidefinite;
(ii) $Ran~B(d+1) \subseteq Ran~M_{d}$ (equivalently,
$B(d+1) = M_{d}W$ for some matrix $W$);
(iii) $ C(d+1) \succeq C^{\flat}:= W^{T}M_{d}W$ (cf. \cite{tcmp1}). \ (Here and in the sequel, for a real symmetric 
matrix $A$, we will write $A \succeq 0$ (resp. $A \succ 0$) to denote that $A$ is positive semidefinite (resp. positive semidefinite 
and invertible).) \ If $M_{d+1} \succeq 0$, then we also have (iv) 
each dependence relation in $Col~M_{d}$ (the column space of $M_{d}$)
extends to $Col~M_{d+1}$. \ In the sequel we say that $M_{d+1}$ is an
$RG$ {\it extension} if properties (i), (ii), and (iv) hold and $M_{d+1}$ is
recursively generated (so, in particular, $(M_d \; B(d+1))$ is recursively generated). \ In the 
sequel, we provide sufficient conditions for $RG$ extensions; note that to verify that
an $RG$ extension is positive semidefinite and recursively
generated, it is only necessary to verify condition (iii). 

For a general $M_{d}$, a significant difficulty in determining the
existence of a flat extension $M_{d+k+1}$ is that there may be infinitely
many positive and recursively generated extensions $M_{d+1}$. \ If one such
extension does not admit a subsequent flat extension, this does not preclude
the possibility that some other extension does. \ In the sequel, we focus
on the class of recursively determinate moment matrices {\it (RD)} introduced in \cite{finitevariety} (cf. \cite{F08}). \ These are
characterized by the property that there can be at most one positive, recursively 
generated extension, and there is a concrete
procedure (described below) for determining the existence or nonexistence of this extension.
 Since such an extension is also recursively determinate, 
we may proceed iteratively to determine the existence or nonexistence
of positive and recursively 
generated
extensions 
\begin{equation} \label{sequence}
M_{d+1},\ldots,M_{2d-1}.
\end{equation}
As we discuss below, the existence of the extensions in (\ref{sequence}) is equivalent to the existence of a flat extension
$M_{d+k+1}$ and, in fact, one of the extensions in (\ref{sequence}) is a flat extension of $M_d$. \ (If $M_{j+1}$ is positive semidefinite, then $M_{j}$ is positive semidefinite 
and recursively generated \cite{tcmp10}, so, using also \cite{BT}, it follows that (\ref{sequence}) is equivalent 
to the existence of a positive semidefinite extension $M_{2d}$.)

A bivariate moment matrix $M_{d}$ admits a block decomposition 
$M_{d} = (B[i,j])_{0\le i,j \le d}$, where  
$$B[i,j] = \bpm
\beta_{i+j,0} & \cdots &  \beta_{i,j} \\
\vdots & \vdots & \vdots \\
\beta_{j,i} & \cdots & \beta_{0,i+j}   
\epm.$$
Thus, $B[i,j]$ is constant on each cross-diagonal; we refer to this as the {\it Hankel property}. \ 
Note that in the extension $M_{d+1}$, $B(d+1) = (B[i,d+1])_{0\le i \le d}$,
and all of the new moments of degree $2d+1$ appear within block
$B[d,d+1]$, either in column $X^{d+1}$ (the leftmost column) or in
column $Y^{d+1}$ (on the right). \ Similarly, all new moments of degree $2d+2$
appear in column $X^{d+1}$ or column $Y^{d+1}$ of $C(d+1)$ ($=B[d+1,d+1]$).
In the sequel, by a {\it{column dependence relation}} we mean a linear dependence
relation of the form $X^{i}Y^{j}=r(X,Y)$, where $deg~r\le i+j$ and each monomial
term in $r$ strictly precedes $x^{i}y^{j}$ in the degree-lexicographic order;
we say that such a relation is {\it{degree reducing}} if $deg~r<i+j$.
A bivariate moment matrix $M_{d}$ is {\it {recursively determinate}} if there are
column dependence relations of the form 
\begin{equation}\label{xrelation}
X^{n} = p(X,Y) \quad (p\in \mathcal{P}_{n-1}, ~n\le d)
\end{equation}
 and
\begin{equation}\label{yrelation}
 Y^{m} = q(X,Y) \quad (q\in \mathcal{P}_{m}, ~q ~\text{has no} ~y^m ~\text{term}, ~ m\le d),
\end{equation}
or with similar relations with the roles of $p$ and $q$ reversed. \ In the sequel, we state the
main results (Theorems \ref{rdext} and \ref{gridthm}) with $p$ and $q$ as in (\ref{xrelation})-(\ref{yrelation}), 
but these results are valid as well with the roles of p and q reversed. \ In Section \ref{Sect2} we show
that if $M_{d}$ is recursively determinate, then the only possible positive, recursively generated (or merely $RG$) extension 
is completely  determined by 
column relations $X^{d+1} = (x^{d+1-n}p)(X,Y)$
and $Y^{d+1}=(y^{d+1-m}q)(X,Y)$. 

The most important case of recursive determinacy occurs when $M_{d}$ is
positive and {\it{flat}}, i.e., $\text{rank}~M_{d} = \text{rank}~M_{d-1}$ (equivalently,
each column of degree $d$ can be expressed as a linear combination of
columns of strictly lower degree). \ A fundamental result of \cite{tcmp1} shows that
in this case $M_{d}$ admits a unique flat extension $M_{d+1}$ (and a
corresponding $\text{rank}~M_{d}$-atomic representing measure). \ In this paper,
we stay within the framework of recursive determinacy, but relax the
flatness condition, and study the extent to which positive, recursively generated
extensions exist. \ 

Our main results are Theorems \ref{rdext} and \ref{RDnew}, which give sufficient conditions for RG extensions, 
and Theorem \ref{gridthm}, which shows that the number of extension steps leading to a flat extension is sometimes proportional 
to the degree of the moment problem. \ Theorem \ref{rdext} shows that if $M_{d}$
is positive and recursively generated, and if all column dependence relations
arise from (\ref{xrelation}) or (\ref{yrelation}) via recursiveness and linearity,
then $M_{d}$ admits a
unique $RG$ extension. \ In general, this extension need not be positive semidefinite
(see the discussion preceding Example \ref{posexample}), but if $d = n+m-2$, then this extension is actually a flat extension, so
$\beta$ admits a representing measure (Corollary \ref{Cblock}). \
Additionally, we show in Theorem \ref{RDnew} that if $M_{d}$ is positive semidefinite,
recursively generated, and recursively determinate, and if all column
dependence relations are degree-reducing, then
$M_{d}$ again admits a unique $RG$ extension. \ However, we show
in Example \ref{notRDexample} that if all of the column relations are degree-reducing
except that $deg~q=m$, then $M_{d}$ need not even admit a block
$B(d+1)$ consistent with recursiveness for $(M_d \; B(d+1))$. \ In Theorem \ref{gridthm} we 
show that for each $d$, there exists  $\beta \equiv \beta^{(2d)}$, with $M_{d}(\beta) \in RD$,
such that in the sequence of positive, recursively generated extensions, $M_{d+1},\ldots,M_{2d-1}$,
the first flat extension is $M_{2d-1}$, so the determination
that a measure exists takes the maximum possible number of extension steps. \ Moreover, at each extension step, $M_{d+i}$ satisfies
the hypotheses of Theorem \ref{rdext}, so it is guaranteed in advance that the next extension $M_{d+i+1}$ 
is well-defined and recursively generated; only its positivity needs to be verified. \
In general, however, the existence of a positive, recursively generated extension
$M_{d+1}$ does not imply the existence of a measure. \ 
In Section \ref{Sect3} we answer \cite[Question 4.19]{finitevariety} by showing that if, 
under the hypotheses of Theorem \ref{rdext}, $M_{d}$ does admit a positive,
recursively generated extension $M_{d+1}$, then $M_{d+1}$ may also satisfy the conditions of Theorem \ref{rdext}, but need
not admit a positive, recursively generated extension $M_{d+2}$, and
thus $M_{d}$ may fail to have a measure.
 
We conclude this section by reviewing and illustrating \cite[Algorithm 4.10]{finitevariety}
concerning extensions of recursively determinate bivariate moment matrices.  \  
We may assume that $M_{d}$ is positive and recursively generated, for otherwise
there is no representing measure. \ (We note that in numerical problems, positivity and
recursiveness can easily be verified using elementary linear algebra.) \
To define block $B(d+1)$ for an extension $M_{d+1}$, note that blocks
$B[0,d+1],\ldots,B[d-1,d+1]$ consist of old moments from $M_{d}$.
To define moments of degree $2d+1$ for block $B[d,d+1]$,
we first use  (\ref{xrelation}) and recursiveness
to define the ``left band" of columns,  $X^{n+i}Y^{d+1-i-n} := (x^{i}y^{d+1-i-n}p)(X,Y)$
($0\le i \le d+1-i$). \ In block $B[d,d+1]$, certain ``new moments"
 in column $X^{n}Y^{d+1-n}$ can be moved up and to the right along cross-diagonals until they reach row $X^{d}$ (the top row of $B[d,d+1]$)
in columns of the ``central band," $X^{n-1}Y^{d+2-n},\ldots,X^{d+2-m}Y^{m-1}$. \ These values
can then be used to define $\langle X^{d+1-m}Y^{m}, X^{d} \rangle$
(the entry in row $X^{d}$, column  $X^{d+1-m}Y^{m}$) by means of
\begin{equation}\label{rightrec}
\langle X^{d+1-m}Y^{m}, X^{d}\rangle := \langle (x^{d+1-m}q)(X,Y), X^{d}\rangle.
\end{equation}
This value may be moved one position down and to the left along its
cross-diagonal and then used to define
$\langle X^{d+1-m}Y^{m}, X^{d-1}Y \rangle
:= \langle (x^{d+1-m}q)(X,Y), X^{d-1}Y\rangle.$
           We repeat this process
successively to complete the definition of column $X^{d+1-m}Y^{m}$ in $B[d,d+1]$ as well as
 the definition of the central band of columns in this block. \
We next complete the definition of $B[d,d+1]$ by successively defining the 
``right band" of columns,
$X^{d-m}Y^{m+1},\ldots,Y^{d+1}$, using
$$X^{d+1-m-i}Y^{m+i}:=(x^{d+1-m-i}y^{i}q)(X,Y)~~~(0\le i\le d+1-m).$$

It is necessary to check that the values in the central and right bands, as just defined,
are compatible with values in the left band, and, more generally, to verify that
$B(d+1)$ is a well-defined moment matrix block. \ If this fails to be the case,
there is no measure. \ If $B(d+1)$ is well-defined, we next check that
$Ran~B(d+1)\subseteq Ran~M_{d}$, for if this is not the case, then there is no
measure. \ Assuming the range condition is satisfied, (\ref{xrelation}) and
(\ref{yrelation}) will hold in the columns of $B(d+1)^{T}$ (the transpose).
We then apply recursiveness and the method used just above in defining $B[d,d+1]$
to attempt to define $C(d+1)\equiv B[d+1,d+1]$. \ Assuming that $C(d+1)$
is well-defined, we further check that $M_{d+1}$ is positive and recursively
generated. \ If any of the preceding steps fails, there is no representing measure.
Our main results (Theorems \ref{rdext} and \ref{RDnew}) show that if all column relations come from
(\ref{xrelation}) or (\ref{yrelation}) via recursiveness and linearity, or if (\ref{xrelation}) - (\ref{yrelation}) hold and all column dependence
relations are degree-reducing, then all of the preceding steps are
guaranteed to succeed, except possibly the positivity of $M_{d+1}$;
thus $M_{d+1}$ is at least an $RG$ extension.

     If $M_{d+1}$, as just defined, is positive and recursively generated, then,
since it is also recursively determinate, we may apply the above procedure 
successively,
in attempting
 to define  positive and recursively generated extensions $M_{d+2}$,
$M_{d+3},\ldots$.
Note that the  central band of degree $d$ in $M_{d}$ has $n+m-d-1$ columns, $X^{n-1}Y^{d-n+1},\ldots,X^{d+1-m}Y^{m-1}$. \
 In each successive extension $M_{d+k}$, the
number of columns  in the central band  of degree $d+k$ is $n+m-d-1-k$.
 Thus, after at most $n+m-d-1$ extension steps, either the extension
process fails, and there is no measure, or the central band disappears and there is 
a flat extension, 
at or before $M_{n+m-1}$, and a measure. \ 
(Note that since $n,m\le d$, this refines our earlier assertion that a flat extension
occurs at or before $M_{2d-1}$.) Another  estimate for the number of
extension steps is based on the {\it{variety}} of $M_{d}$, defined as
$\mathcal{V} \equiv \mathcal{V}(M_{d}) := \displaystyle \bigcap \limits_{r\in \mathcal{P}_{d},r(X,Y)=0}
\mathcal{Z}_{r}$, where $\mathcal{Z}_{r}$ is the set of real zeros of $r$.
It follows from \cite{finitevariety} that the number of extension steps leading to a flat extension
is at most $1+ \text{card}~\mathcal{V} - \text{rank}~M_{d}$. \ Note also that when a measure exists, it is supported inside 
$\mathcal{V}$ \cite{tcmp10}, so its support is a subset of the finite real variety determined by $x^{n}-p(x,y)$ and $y^{m}-q(x,y)$.

Examples  are known where the $RG$ extension $M_{d+1}$ is not positive semidefinite
(cf. \cite[Example 4.18]{finitevariety}, [CFM, Theorem 5.2], both with $n=m=d=3$, and the example of Section \ref{Sect3} 
(below), with $d=5$, $n=m=4$). \ We next present an example, adapted from \cite[Example 5.2]{F08}, which illustrates
the algorithm in a case leading to a measure.

\begin{example}\label{posexample}
Let $d=3$ and consider
$$M_{3}=
  \left( \begin{array}{cccccccccc}
  1 & 0 & 0 & 1 & 2 & 5 & 0 & 0 & 0 & 0 \\
  0 & 1 & 2 & 0 & 0 & 0 & 2 & 5 & 14 & 42      \\
  0 & 2 & 5 & 0 & 0 & 0  & 5 & 14 & 42 & 132 \\
  1 & 0 & 0 & 2 & 5 & 14 & 0  & 0 & 0 & 0 \\
  2 & 0 & 0 & 5 & 14 & 42 & 0 & 0 & 0 & 0 \\
  5 & 0 & 0 & 14 & 42 & 132 & 0 & 0 & 0 & 0 \\
  0 & 2 & 5 & 0 & 0 & 0  & 5 & 14 & 42 & 132 \\
  0 & 5 & 14 & 0 & 0 & 0 & 14 & 42 & 132 & 429  \\
  0 & 14 & 42 & 0 & 0 & 0 & 42 & 132 & 429 & c \\
  0 & 42 & 132 & 0 & 0 & 0 & 132 & 429 & c & d
    \end{array}  \right).
$$
We have $M_{3}\succeq 0$, $M_{2}\succ 0$, and $\text{rank}~M_{3} = 8 \Longleftrightarrow d=2026881 - 2844 c + c^2$. \ 
When $\text{rank}~M_{3}=8$, then the two
column relations are
$$ Y = X^{3}$$
and
$$ Y^{3} = q(X,Y),$$
where $q(x,y):=(5715 - 4 c) x+ 10 (-1428 + c) y - 3 (-2853 + 2 c) x^2 y + (-1422 + c) x y^2$. \
Let $r_{1}(x,y) = y - x^{3}$ and $r_{2}(x,y) =
y^{3}-q(x,y)$. \ With these two column relations in hand, Theorem \ref{rdext} guarantees the existence of a unique 
$RG$ extension $M_{4}$. \ To test the positivity of $M_{4}$, we calculate the determinant of the $9 \times 9$ matrix consisting of the rows and
columns of $M_{4}$ indexed by the monomials $1,x,y,x^2,xy,y^2,x^2y,xy^2,x^2y^2$. \ A straightforward calculation using {\it Mathematica} shows 
that three cases arise:

(i) $c<1429$: here $M_{4} \not \succeq 0$, so $M_{3}$ admits no representing measure;

(ii) $c=1429$: here $M_{4}$ is a flat extension of $M_{3}$, so by the main result in \cite {tcmp1}, $M_{3}$ admits an $8$-atomic 
representing measure;

(iii) $c>1429$: here $M_{4}$ is a positive RG extension of $M_{3}$ with rank $9$. \ Although $M_{4}$ is not a flat extension of $M_{3}$, it nevertheless satisfies the hypotheses of Theorem \ref{rdext}, so Corollary \ref{flatext} implies that $M_{4}$ admits a flat extension $M_{5}$, and therefore $M_{3}$ has a $9$-atomic representing measure. \ Moreover, since the original algebraic variety $\mathcal{V} \equiv \mathcal{V} (M_{3})$ associated with $M_{3}$, $\mathcal{Z}_{r_{1}}\bigcap \mathcal{Z}_{r_{2}}$, can have at most $9$ points (by B\'ezout's Theorem), it follows that $\mathcal{V}=\mathcal{V}(M_{5})$.  This algebraic variety must have exactly $9$ points, and thus constitutes the support of the unique representing measure for $M_{3}$. \ 

To illustrate this case, we take the special value $c=1430$, so that $q(x,y) \equiv -5x+20y-21x^2y+8xy^2$. 
\ Let $\alpha:=\frac{1}{2}\sqrt{5-2\sqrt{5}}$ and $\gamma:=\sqrt{5}\alpha$. \ A calculation shows that
$\mathcal{V} =  \{(x_{i},x_{i}^{3})\}_{i=1}^{9}$,
where $x_{1} = 0$, $x_{2} =\frac{1}{2}(-1-\sqrt{5}) \approx -1.618$, $x_{3} =\frac{1}{2}(1-\sqrt{5}) \approx -0.618$, $x_{4} 
=-x_{3} \approx 0.618$, 
$x_{5} =-x_{2} \approx 1.618$, 
$x_{6} =-\alpha-\gamma \approx -1.176$,
$x_{7} =-\alpha+\gamma \approx 0.449$,
$x_{8} =-x_{7} \approx -0.449$ and
$x_{9} =-x_{6} \approx 1.176$. \ $M_{3}$ satisfies the hypothesis of Theorem \ref{rdext} with $n=m=3$,
so we proceed to generate the $RG$ extension
$M_{4}$. \ This extension is uniquely determined by imposing the
column relations $X^{4} = XY$, $X^{3}Y = Y^{2}$,
$XY^{3} = (xq)(X,Y)$, and $Y^{4}= (yq)(X,Y)$
(first in 
 $ \left( \begin{array}{cc}
  M_{3} & B(4) \\
    \end{array}  \right)  $, then in
 $ \left( \begin{array}{cc}
  B(4)^{T} & C(4) 
    \end{array}  \right)  $). \
 A calculation shows
that, as expected, these relations unambiguously define a positive moment matrix
$M_{4}$ with $\text{rank}~M_{4} = 9$ ($>8=\text{rank}~M_{3}$). \ It follows that $M_{3}$ admits no
flat extension $M_{4}$, so we proceed  to construct the $RG$
 extension $M_{5}$, uniquely determined by imposing the relations
$X^{5} = X^{2}Y$, $X^{4}Y = XY^{2}$, $X^{3}Y^{2} = Y^{3}$,
$X^{2}Y^{3} = (x^{2}q)(X,Y)$, $XY^{4}= (xyq)(X,Y)$,
$Y^{5} = (y^{2}q)(X,Y)$. \ A calculation of these columns
(first in 
 $ \left( \begin{array}{cc}
  M_{4} & B(5) \\
    \end{array}  \right)  $, then in
 $ \left( \begin{array}{cc}
  B(5)^{T} & C(5) 
    \end{array}  \right)  $),
shows that, as again expected, they do fit together to unambiguously define
a moment matrix $M_{5}$.
From the form of $q(x,y)$, we see that $M_{5}$ is actually a
flat extension of $M_{4}$, in keeping with the above discussion. \ Corresponding to this flat extension is the unique, 
$9$-atomic, representing
measure $\mu\equiv \mu_{M_{5}}$ as described in \cite{tcmp10}. \
Clearly, $\text{supp}~\mu = \mathcal{V}$, so $\mu$ is of the form
$\mu = \sum_{i=1}^{9} \rho_{i}\delta_{(x_{i},x_{i}^{3})}$. \
To compute the densities, we use the method of \cite{tcmp10} and find
$\rho_{1}=\frac{1}{5}=0.2$, $\rho_{2}=\rho_{5}=\frac{-1+\sqrt{5}}{8\sqrt{5}} \approx 0.069$, 
$\rho_{3}=\rho_{4}=\frac{1+\sqrt{5}}{8\sqrt{5}} \approx 0.181$, $\rho_{6}=\rho_{9}=\frac{5+3\sqrt{5}}{40\sqrt{5}} \approx 0.131$, and $\rho_{7}=\rho_{8}=\frac{-5+3\sqrt{5}}{40\sqrt{5}} \approx 0.019$. 
\ Thus, the existence of a representing measure for $\beta^{(6)}$
is established on the basis of the extensions $M_{4}$ and
$M_{5}$,
 in keeping with Theorem \ref{rdext}.
Note  that in this case, the actual number of extensions leading to
a flat extension can be computed as either $n+m-d-1$ or as
$1+ \text{card}~\mathcal{V}-\text{rank}~M_{3}$, which is consistent with our earlier discussion. 
$\square$
\end{example}

\textit{Acknowledgment}. \ Examples \ref{posexample} and \ref{notRDexample}, and the example of Section \ref{Sect3}, were obtained using
calculations with the software tool \textit{Mathematica \cite{Wol}}. \ 


\section{The extension of a bivariate  $RD$ positive moment matrix }\label{Sect2}

\setcounter{equation}{0}
In Theorem \ref{rdext} (below) we show that a positive recursively determinate
moment matrix $M_{d}\equiv M_{d}(\beta)$,
each of whose column dependence relations is recursively generated by a relation
of the form  
\begin{equation}\label{X} X^{n} = p(X,Y), \; (p\in \mathcal{P}_{n-1})
\end{equation}
or
\begin{equation}\label{Y}
         Y^{m} = q(X,Y), \; (q\in \mathcal{P}_{m}),
\end{equation}
(where $n,m\le d$ are fixed and $q$ has terms $x^{u}y^{v}$ with $v<m$),
 always admits 
a unique $RG$ extension
$$M_{d+1}\equiv \bpm
M_{d} &  B(d+1) \\
B(d+1)^{T} & C(d+1) \epm .
$$
The main step towards Theorem \ref{rdext} is the following result, which shows
that $M_{d}$ (as above) admits
an extension block
$B(d+1)$ that is consistent with the structure of a positive,
recursively generated moment matrix extension $M_{d+1}$.

\begin{theorem}\label{main}
Suppose the bivariate moment matrix
$M_{d}(\beta)$ is positive and recursively generated,
with column dependence relations generated entirely by 
(\ref{X}) and (\ref{Y}) via recursiveness and linearity. \ 
Then there exists a unique moment matrix block
$B(d+1)$ such that
$\bpm
M_{d} &  B(d+1)
\epm$ is recursively generated and $Ran~B(d+1)\subseteq Ran~ M_{d}$.
\end{theorem}

 The hypothesis implies that the column dependence relations in $M_{d}$
are precisely those of the form
\begin{equation}\label{Xrec}
X^{n+i}Y^{j} = (x^{i}y^{j}p)(X,Y) \quad (i,j\ge 0,~i+j+n\le d)
\end{equation}
and
\begin{equation}\label{Yrec}
X^{k}Y^{m+l}=(x^{k}y^{l}q)(X,Y) \quad (k,l\ge 0, ~k+l+m\le d).
\end{equation}
In particular, the degree $d$ columns $X^{d},\ldots,X^{n}Y^{d-n}$ 
 are recursively determined in terms of
columns of strictly lower degree. \ Since, by (\ref{Yrec}), each column 
 $X^{d-m-k}Y^{m+k}$ ($0\le k\le d-m$) may be
expressed as a  linear combination of columns to its left, it follows
that if $n\le d-m+1$, then $M_{d}$ is flat. \
Since a flat positive moment matrix
admits a unique positive, recursively generated extension (cf. \cite{tcmp1}), we may
assume that not every column of degree $d$ is recursively determined,
i.e., $n>d-m+1$, or
\begin{equation}\label{n+m}
n+m > d+1 .
\end{equation}

We may denote
\begin{equation}\label{X^n}
X^{n} = p(X,Y) \equiv \displaystyle \sum \limits_{r,s\ge 0, r+s\le n-1}
a_{rs}X^{r}Y^{s}
\end{equation}
and
\begin{eqnarray}
Y^{m} &=&q(X,Y)\equiv \sum\limits_{u,v\geq 0,u+v\leq m,v<m}b_{uv}X^{u}Y^{v} \nonumber \\
&\equiv &\sum\limits_{a,b\geq 0,a+b\leq m-1}\alpha
_{ab}X^{a}Y^{b}+\sum\limits_{c,e\geq 0,c+e=m,e<m}\gamma _{ce}X^{c}Y^{e}.
\end{eqnarray}
Thus, in any positive and recursively generated (or merely $RG$) extension
$M_{d+1}$, certain columns of
$B(d+1)$ are recursively determined.
On the left of $B(d+1)$,
 there is a band of
columns,
\begin{eqnarray}\label{Xton_rec}
X^{n+f}Y^{d+1-n-f} &:=& (x^{f}y^{d+1-n-f}p)(X,Y) \nonumber \\
&\equiv& \sum \limits_{r,s\ge 0, r+s\le n-1}^{}a_{rs}X^{r+f}Y^{s+d+1-n-f} \quad (0\le f\le d+1-n)
\end{eqnarray}
each of which is well-defined as a linear combination
of columns of $M_{d}$. \ On the right of $B(d+1)$
there is another band of recursively determined columns,
\begin{eqnarray}\label{Y^m_rec}
X^{d+1-m-g}Y^{m+g} &:=& (x^{d+1-m-g}y^{g}q)(X,Y) \nonumber \\
&\equiv& \sum \limits_{u,v\ge 0, u+v\le m, v<m}^{}b_{uv}X^{u+d+1-m-g}Y^{v+g} \quad (0\le g\le d+1-m). \quad \quad
\end{eqnarray}
If $deg~q=m$, the sum in (\ref{Y^m_rec}) may involve
columns from the middle band,
$X^{u+d+1-m-g}Y^{v+g}$ ($u+v=m, u+d+1-m-n < g < m-v$),
which has not yet been defined, so some care is needed in
implementing (\ref{Y^m_rec}). 

The proof of Theorem \ref{main} entails two main steps, which we prove in detail in Section \ref{PROOF}: the construction of the 
block $B(d+1)$, and the verification of the inclusion $Ran~B(d+1) \subseteq Ran~M_d$. \ Assuming that we have already built a unique block $B(d+1)$ 
consistent with the existence of a positive, recursively generated extension
\begin{equation*}
M_{d+1}\equiv
\bpm
M_{d}  &  B(d+1)     \\
B(d+1)^{T}  &  C(d+1)  
\epm,
\end{equation*}
we next use this to construct a unique block $C(d+1)$ consistent with the existence of an $RG$ extension. 

\begin{cor}\label{Cblock}
If $M_{d}$ satisfies the hypotheses of Theorem \ref{main}, then
there exists a unique moment matrix block $C\equiv C(n+1)$
consistent with the structure of an $RG$ extension $M_{d+1}$.
\end{cor}
\begin{proof}
In any $RG$ extension 
$M_{d+1}$
the column relations (\ref{Xton_rec}) and (\ref{Y^m_rec}) must hold. \ The proof of Theorem \ref{main} shows that these relations define a unique moment matrix block $B\equiv B(d+1)$ consistent
with positivity and recursiveness.
To define $C\equiv C(n+1)$, we may formally repeat the proof of
Theorem \ref{main} concerning the well-definedness and uniqueness of block
$B(d+1)$, but applying the argument with $M_d$ replaced with $B(d+1)^{T}$,
and $B[d,d+1]$ replaced with $C\equiv B[d+1,d+1]$. \ In brief, we use
$B(d+1)^{T}$ and (\ref{Xton_rec}) to define the left recursive
band in $C$. \ We then define column $X^{d+1-m}Y^{m}$ 
by applying  (\ref{Y^m_rec}) successively, 
starting in row $X^{n+1}$, so that this column is Hankel with respect to the
central band, which we are completing simultaneously. \ We then use (\ref{Y^m_rec}) to successively define the remaining columns
on the right. \ Lemma \ref{hankel} can be used to show that the left band is internally
Hankel, and an adaptation of the argument in Lemma \ref{Hankellemma} can be used to show
that column $X^{d+1-m}Y^{m}$ is Hankel with respect to the left and
central blocks. \ Finally, the argument of Lemma \ref{righthankel} can be adapted to show that
that the right band is also Hankel. 
\end{proof}

By combining Theorem \ref{main} with Corollary \ref{Cblock}, we immediately obtain the first of our main results, which follows.

\begin{theorem}\label{rdext} If $M_{d}$ is positive, with column relations generated entirely by (\ref{X}) and (\ref{Y}) 
via recursiveness and linearity, then $M_{d}$ admits a unique $RG$ extension $M_{d+1}$, i.e., $Ran~B(n+1)\subseteq Ran~M_{d}$, (\ref{Xton_rec})-(\ref{Y^m_rec}) hold 
in $Col \; M_{d+1}$, and $M_{d+1}$ is recursively generated.
\end{theorem}

\begin{cor}\label{flatext}
If $M_{d}$ satisfies the hypotheses of Theorem \ref{rdext} and $d=n+m-2$, then
$M_{d}$ admits a flat moment matrix extension $M_{d+1}$ (and
$\beta$ admits a $\text{rank}~M_{d}$-atomic representing measure).
\end{cor}
\begin{proof}
Each column in the left band is, from (\ref{Xton_rec}), a linear combination
of columns of strictly lower degree.
Since $d=n+m-2$, there is no central band in the construction of $B(d+1)$ in Theorem \ref{main} and
of $C(d+1)$ in Corollary \ref{Cblock}. \ It thus follows
from (\ref{Y^m_rec}) that
each column in the right band is also a linear combination of columns
of strictly lower degree, so $M_{d+1}$ is a flat extension.
\end{proof}

To illustrate Corollary \ref{flatext} in the simplest case, let $n=m=d=2$ and suppose that $M_2$ satisfies the hypotheses of Theorem \ref{rdext}.
\ It follows from \cite{tcmp6} that $M_{2}$
admits a representing measure if and only if the equations
$x^{2}-p(x,y) = 0$ and $y^{2}-q(x,y)=0$ have at least 4 common real zeros.
Corollary \ref{flatext} implies that the latter ``variety condition" is superfluous; indeed, from Corollary \ref{flatext}, there {\it{is}} a representing measure, so \cite{tcmp6} implies that the system {\it{must}} have at least 4 ($= \text{rank}~M_{2}$) common real zeros.

Note that if $M_{d}(\beta)$ satisfies the hypothesis of Theorem \ref{rdext}, then the existence or nonexistence of a representing measure for $\beta$ will be established in at most $d-1$ extension steps (after which the central band would vanish and every column of $M_{2d-1}$ would be recursively determined). \ The next result shows that for every $d\ge 2$, there exists $M_{d}(\beta)$, satisfying the conditions of Theorem \ref{rdext}, for 
which the determination that a representing measure exists entails the maximum number of extension steps, each of which falls within the scope of Theorem \ref{rdext}. \ 
\begin{theorem}\label{gridthm}
For $d \ge 1$, there exists a moment matrix $M_{d}$, satisfying the conditions of Theorem \ref{rdext}, for which 
the extension algorithm determines successive positive, recursively generated extensions $M_{d+1},\ldots,M_{2d-1}$, and for which the first 
flat extension occurs at $M_{2d-1}$. \ Moreover, each extension $M_{d+i}$ satisfies the conditions of Theorem \ref{rdext}, so
to continue the sequence it is only necessary to verify that the $RG$ extension $M_{d+i+1}$ is positive semidefinite.
\end{theorem}

\begin{remark} \label{newrmk}
To illustrate the significance of Theorem \ref{gridthm}, let us compare it to the following result of \cite[Theorem 1.2]{tcmp9}: 
If $M_d(\beta)$ is a
bivariate moment matrix with a column relation $p(X,Y) = 0 ~ (deg ~ p \le 2)$, then $\beta$ has a measure if and only if $M_d$ is
positive, recursively generated, and $\text{rank} ~ M_d \le \text{card} ~ \mathcal{V}(M_{d})$. \ In this result, we see that the existence of a measure can
be determined directly from the data by establishing the positivity, rank, and variety of $M_d$. \ By contrast, in Theorem \ref{gridthm}
we see that it may be necessary to extend $M_d$ to $M_{2d-1}$ in order to establish that a measure exists. \ In this sense, within
the framework of moment matrices, we see that the general case of the truncated moment problem cannot be solved in ``closed form."
\ We may therefore seek to go beyond the framework of moment matrices. \ Recall that for $\beta \equiv \beta^{(2d)}$, 
$L_{\beta}$ is {\it{positive}} if $p\in \mathcal{P}_{2d}$, $p|_{R^d} \ge 0 \Longrightarrow L_{\beta}(p) \ge 0$. \ In \cite{tcmp12} 
we showed that $\beta$ admits a representing measure if and only if $L_\beta$ admits a positive extension 
$L:\mathcal{P}_{2d+2} \longrightarrow \mathbb{R}$. \ Thus, as an alternative to constructing all of the extensions $M_{d+1},\cdots, M_{2d+1},$ 
in principle it would suffice to test the positivity of the Riesz functional corresponding to $M_{d+1}$. \ Unfortunately, at present there 
is no known concrete test for positivity of Riesz functionals (except in special cases, cf. \cite{tcmp12}, \cite{FN1}, \cite{FN2}), 
so the moment matrix extension algorithm remains the most viable approach to resolving the existence of a representing measure in 
the bivariate $RD$ case.
\end{remark}

For the proof of Theorem \ref{gridthm}, we require some preliminaries. \ For $d\ge 1$, suppose 
$x_{1},\ldots, x_{d}$ are distinct and  $y_{1},\ldots, y_{d}$ are distinct. \ Let $P(x,y) := (x-x_{1})\cdots (x-x_{d})$,  $Q(x,y):=(y-y_{1})\cdots (y-y_{d})$, and set $\mathcal{Z}_{P,Q}:= \{(x_{i},y_{j})\}_{1\le i,j \le d}$, the common zeros of $P$ and $Q$. \ 
Let  $J$ be an  ideal in $\mathbb{R}[x,y]$ with real variety $\mathcal{V}\equiv
\mathcal{V}(J):= \{(x,y)\in \mathbb{R}^{2}:s(x,y) = 0~ \forall s\in J\}$. \ Let $I(\mathcal{V})
=\{f \in \mathbb{R}[x,y]: f|\mathcal{V} \equiv 0 \}$. \ In general, $I(\mathcal{V}(J))$ may be strictly larger than $J$ \cite{CLO}. \
However, for $J:=(P,Q)$ (with $P$ and $Q$ as above), we will show below (Proposition \ref{prop 213}) that 
each element of $I(\mathcal{V}(J))$
admits a ``degree-bounded" representation which displays it as a member of
 $J$; in particular,  $J$ is a {\it{real ideal}}
in the sense of \cite{Mo}. \ Although this result may well be known, we could not find a reference, so we
include a proof for the sake of completeness. \ First, we need three auxiliary results.

\begin{lemma} \label{divisionalg} (The Division Algorithm in $\mathbb{R}[x_{1},\cdots ,x_{n}]$ \cite[Section 2.3, Theorem 3]{CLO}) \ 
Fix a monomial order $>$ on $\mathbb{Z}_{\geq 0}^{n}$ and let $F=(f_{1},\cdots ,f_{s})$ be an ordered $s$-tuple of
polynomials in $\mathbb{R}[x_{1},\cdots ,x_{n}]$. \ Then every $f\in \mathbb{R}[x_{1},\cdots ,x_{n}]$ can be written as 
\begin{equation*}
f=a_{1}f_{1}+\cdots +a_{s}f_{s}+r,
\end{equation*}%
where $a_{i},\in \mathbb{R}[x_{1},\cdots ,x_{n}]$, and either 
$r=0$ or $r$ is a linear combination, with coefficients in $\mathbb{R}$, of
monomials, none of which is divisible by any of the leading terms in $f_{1},\cdots ,f_{s}$.

Furthermore, if $a_{i}f_{i}\neq 0$, then we have $\text{multideg}~(f)\geq \text{multideg}~(a_{i}f_{i})$.
\end{lemma}

\begin{lemma} \cite[p. 67]{Sau} \label{Sauer} \ For $N \geq 1$ let ${v_1,\cdots,v_N}$ be distinct points in 
$\mathbb{R}^2$, and consider the multivariable
Vandermonde matrix $V_{N}:=(v_i^{\alpha})_{1 \le i \le N, \alpha \in \mathbb{Z}_+^2, \left|\alpha\right| \leq N-1}$, of size $N \times \frac{N(N+1)}{2}$. \ Then the rank of $V_N$ equals $N$.
\end{lemma}

\begin{corollary} \label{vander} \ Let ${\bf x} \equiv \{ x_1,\ldots,x_m \} $ and ${\bf y} \equiv \{y_1,\ldots,y_n\}$ 
be sets of distinct real numbers, and consider the 
grid ${\bf x} \times {\bf y} := \{(x_i,y_j)\}_{1 \leq i \leq m, 1 \leq j \leq n}$ consisting of $N:=mn$ distinct points in 
$\mathbb{R}^2$. \ Then the generalized Vandermonde matrix $V_{{\bf x} \times {\bf y}}$, obtained from $V_{N}$ by removing all columns indexed 
by monomials divisible by $x^m$ or $y^n$, is invertible.
\end{corollary}

\begin{proof}
The columns of $V_{N}$ are indexed by the monomials in $x$ and $y$ of degree at most $N$, listed in degree-lexicographic order. \ 
The size of $V_{N}$ is $N \times \frac{N(N+1)}{2}$, and by Lemma \ref{Sauer} we know that its rank is $N$. \ 
We will show that $V_{{\bf x} \times {\bf y}}$ has exactly $N$ columns, and that each column that was removed from $V_{N}$ to 
produce $V_{{\bf x} \times {\bf y}}$ is a linear combination of other columns in $V_{N}$. \ 
Toward the first assertion, assume without loss of generality that $m \leq n$, let $k:=n-m$ (so that $m+k=n$), and 
observe that the columns of $V_{{\bf x} \times {\bf y}}$ are 
indexed by the following monomials: 

\begin{eqnarray*}
1,\\
x,y,\\
x^2,xy,y^2,\\
\cdots,x^{m-1},\cdots,y^{m-1}, \\
x^{m-1}y,\cdots,xy^{m-1},y^m, \\
x^{m-1}y^2,\cdots,xy^{m},y^{m+1}, \\
x^{m-1}y^3,\cdots,xy^{m+1},y^{m+2}, \\
\cdots, \\
x^{m-1}y^k, \cdots, xy^{m+k-2},y^{m+k-1}, \\
x^{m-1}y^{k+1},\cdots,xy^{m+k-2} \\
\cdots, \\
x^{m-1}y^{n-1}. 
\end{eqnarray*}
The number of monomials is then $(1+2+\cdots+m)+mk+[(m-1)+(m-2)+\cdots+2+1]=\frac{m(m+1)}{2}+mk+\frac{(m-1)m}{2}
=m^2+mk=m(m+k)=mn$. \ It follows that $V_{{\bf x} \times {\bf y}}$ has exactly $N \equiv mn$ columns.

To prove the second assertion, observe that the polynomials $P:=(x-x_1)\cdots (x-x_m)$ and $Q:=(y-y_1)\cdots (y-y_n)$ vanish identically on 
${\bf x} \times {\bf y}$, and therefore the columns of $V_{N}$ indexed by multiples of $x^m$ or $y^n$ are linear combinations of 
columns preceding them in degree-lexicographic order.

By combining the preceding two assertions, it follows that $V_{\bf{x} \times \bf{y}}$, having size $N$ and rank $N$, must be invertible.
\end{proof}

The following result is a special case of Alon's Combinatorial Nullstellensatz \cite{A}; for completeness, we give a proof based on
Corollary \ref{vander}.

\begin{corollary} \label{md-1}  
Let $G \equiv {\bf x} \times {\bf y}$ be a grid as in Corollary \ref{vander}, let $N:=mn$, and let $p \in \mathbb{R}[x,y]$ be such that 
$\text{deg}_x~p<m$ and $\text{deg}_y~p<n$. \ Assume also that $p|_G \equiv 0$. \ Then $p \equiv 0$.
\end{corollary}

\begin{proof}
We wish to apply Corollary \ref{vander}. \ From the hypotheses, it is straightforward to verify that $p$ does not contain any 
monomials divisible by $x^m$ or $y^n$, so $\hat{p}$, properly extended with zeros 
to indicate the absence of relevant monomials, can be regarded as a vector in $\mathbb{R}^N$, the domain of the 
generalized Vandermonde matrix $V_{G}$ in Corollary \ref{vander}. \ Since, by assumption, $p(x_i,y_j)=0$ for all $1 \le i \le m$ and $1 \le j \le n$, 
it follows that $V_{G}\hat{p}=0$. \ Since $V_{G}$ is invertible (by Corollary \ref{vander}), we must have $\hat{p}=0$, so $p \equiv 0$, as desired.
\end{proof}

\begin{prop} \label{prop 213} \ Let $P(x,y) := (x-x_{1})\cdots (x-x_{d})$ and let $Q(x,y):=(y-y_{1})\cdots (y-y_{d})$. \ 
If $\rho := \text{multideg}~(f) \geq d$ and $f| \mathcal V((P,Q)) \equiv 0$, then there exists $u,v \in \mathcal{P}_{\rho-d}$ such that $f=uP+vQ$.
\end{prop}

\begin{proof} 
Let $\mathcal{V} := \mathcal{V}((P,Q))$. \ By Lemma \ref{divisionalg}, we can write $f=uP+vQ+r$, where $\text{multideg}~ (uP)\leq \rho$ and $\text{multideg}~ (vQ) \leq \rho$. \ It follows that $u,v \in \mathcal{P}_{\rho-d}$ and that $r| \mathcal{V} \equiv 0$. \ Moreover, $r$ is a linear combination, with coefficients in $\mathbb{R}$, of
monomials, none of which is divisible by any of the leading terms in $P$ and $Q$, that is, they are not divisible by $x^d$ and $y^d$. \ Therefore, 
$r$ satisfies the hypotheses of Corollary \ref{md-1} with $m=n=d$. \ By Corollary \ref{md-1}, $r \equiv 0$. \ Thus, $f=uP+vQ$, as desired.
\end{proof}

\begin{proof}[Proof of Theorem \ref{gridthm}]
At several points of the proof we will use the fact that if a moment matrix $M_{k}$ admits a representing measure $\nu$ and $f \in \mathcal{P}_{k}$,
then $f |_{\text{supp}~ \nu} \equiv 0$ if and only if $f(X,Y)=0$ in $\mathcal{C}_{M_{k}}$ \cite[Proposition 3.1]{tcmp1}. \ Let 
$x_1,\ldots,x_d$ and $y_1,\ldots,y_d$ be sets of distinct real numbers, and let $G:={\bf x} \times {\bf y} \equiv {(x_i,y_j)}_{1 \le i,j \le d}$ 
denote the corresponding grid. \ Let $\mu$ denote a measure whose support is precisely equal to $G$ and let $M_{d}:= M_{d}[\mu]$. \ 
Let $P(x,y) := (x-x_{1})\cdots (x-x_{d})$ and let $Q(x,y):=(y-y_{1})\cdots (y-y_{d})$. \ Since $P |_G \equiv 0$ and $Q |_G \equiv 0$, then $P(X,Y)=0$ and $Q(X,Y)=0$ in $\mathcal{C}_{M_d}$, whence
$X^d=p(X)$ and $Y^d=q(Y)$ for certain $p,q \in \mathcal{P}_{d-1}$ satisfying $P(x,y) \equiv x^d-p(x)$ and $Q(x,y) \equiv y^d-q(y)$; 
thus, $M_d$ is recursively determinate. 
\ We first show that the only column dependence relations in $M_d$ arise from the above relations 
via linearity, so that $M_d$ falls within the scope of Theorem \ref{rdext}. \ If $\text{deg} ~ f=d$ and $f(X,Y)=0$ in $Col ~ M_d$, 
then $f |_G \equiv 0$, so Proposition \ref{prop 213} implies that there exists scalars $u$ and $v$ such that $f=uP+vQ$. \ Thus, 
$f(X,Y)=uP(X,Y)+vQ(X,Y)$. \ Further, if $\text{deg}~ f<d$ and $f(X,Y)=0$, then since $f|_{G} \equiv 0$, 
it follows from Corollary \ref{md-1} that $f \equiv 0$ (whence $M_{d-1} \succ 0)$. \ Thus, $M_d$ satisfies the
 conditions of Theorem \ref{rdext}. \ 
 
 Since $M_{d}$ has the finitely atomic representing measure $\mu$, $M_{d}$ admits successive positive, recursively generated extensions 
 $M_{d+1}[\mu], M_{d+2}[\mu],\ldots$, so clearly these are the unique successive positive, recursively determined extensions of $M_{d}$; let 
 $M_{d+k} := M_{d+k}[\mu] ~ (1 \le k \le d-1)$. \ We seek to show that each of $M_{d+1},\ldots,M_{2d-1}$ falls within the scope of Theorem \ref{rdext} and that the first flat extension in this sequence occurs with $\text{rank}~M_{2d-1}= \text{rank}~M_{2d-2}$. \ 
We first give a concrete description of $ker ~ M_{d+k}$. \ Since $M_{d-1} \succ 0$, if $r \in \mathcal{P}_{d+k}$ with $\hat{r} \in 
ker ~ M_{d+k}$, then $\text{deg} ~ r=d+j$ for $0 \le j \le k$. \ Since $\mu$ is a representing measure for $M_{d+k}$, it follows that $r|\text{supp}~\mu \equiv 0$. \ Proposition \ref{prop 213} now implies that there exist $u,v\in \mathcal{P}_{j}$ such that $r=uP + vQ$ (with $P$ and $Q$ defined above in the description of $\mu$). \ Thus $ker~M_{d+k}$ is indexed by the recursively determined columns; precisely, $ker~M_{d+k}$ is the span of all of  the columns $\widehat{x^{s}y^{t}(x^d-p)}$ and
 $\widehat{x^{s}y^{t}(y^d-q)}$ ($s,t\ge0$, $s+t\le k$). \ Thus, $M_{d+k}$ satisfies the conditions of Theorem \ref{rdext}. \ 
 In passing from $M_{d+k-1}$ to $M_{d+k}$ there are $d+k+1$ new columns, of which $2(k+1)$ are recursively determined, and since these correspond (as just above) to elements of $ker~M_{d+k}$, we have
$\text{rank}~M_{d+k}= \text{rank}~M_{d+k-1} + (d+k+1)-2(k+1)= \text{rank}~M_{d+k-1}+ d-k-1$. \ Thus the first flat extension occurs when $k=d-1$, in passing from $M_{2d-2}$ to $M_{2d-1}$.
\end{proof} 

We continue with an example which shows that Theorem \ref{main}
is no longer valid if we permit column dependence relations in $M_{d}$ in addition
to those in (\ref{Xrec}) - (\ref{Yrec}).
\begin{example}\label{notRDexample}
We define $M_{3}$
 by setting $\beta_{00}= \beta_{20}=\beta_{02} = 1$;
$\beta_{11}= \beta_{30}=\beta_{21}=\beta_{03}=0$;
$\beta_{12}=\beta_{40}=2$; $\beta_{31}=\beta_{13}=0$;
$\beta_{22}=5$, $\beta_{04}=22$; 
$\beta_{50}=-1$, $\beta_{41}=-2$, $\beta_{32}=13$,
$\beta_{23}=3$, $\beta_{14}=\frac{894}{13}$,
$\beta_{05}= \frac{336}{13}$;
$\beta_{60}=178$, $\beta_{51}=139$, $\beta_{42}=159$,
$\beta_{33}= \frac{1657}{13}$, $\beta_{24}= \frac{4298}{13}$,
$\beta_{15}=r$, $\beta_{06}= \gamma
\equiv \frac{443272376768-2742712830r-4826809r^{2}}{41327767}$.
Thus, we have
\begin{equation}\label{m3ex}
M_{3} =
\bpm
1  &  0 & 0 & 1 & 0 & 1 & 0 & 0 & 2 & 0   \\
0  &  1 & 0 & 0 & 0 & 2 & 2 & 0 & 5 & 0   \\
0  &  0 & 1 & 0 & 2 & 0 & 0 & 5 & 0 & 22   \\
1  &  0 & 0 & 2 & 0 & 5 & -1 & -2 & 13 & 3   \\
0  &  0 & 2 & 0 & 5 & 0 & -2 & 13 & 3 & \frac{894}{13}   \\
1  &  2 & 0 & 5 & 0 & 22 & 13 & 3 &     \frac{894}{13}  &    \frac{336}{13}    \\
0  &  2 & 0 & -1 & -2 & 13 & 178 & 139 & 159 &      \frac{1657}{13}   \\
0  &  0 & 5 & -2 & 13 & 3 & 139 & 159 &  \frac{1657}{13} &\frac{4298}{13}\\                      
2  &  5 & 0 & 13 & 3 & \frac{894}{13} & 159 & \frac{1657}{13} & \frac{4298}{13} &
 r   \\
0  &  0 & 22 & 3 &  \frac{894}{13}&  \frac{336}{13} & \frac{1657}{13}   
                  & \frac{4298}{13} & r & \gamma    
\epm.
\end{equation}
It is straightforward to check that $M_{3}$ is positive, recursively generated, and recursively determinate, with $M_{2}\succ 0$,
$\text{rank}~M_{3} = 7$ and column
dependence relations
\begin{equation}\label{X3}
X^{3}= p(X,Y):= 40\cdot 1 -24X+4Y-53X^{2}-2XY+13Y^{2},
\end{equation}
\begin{equation}\label{X2Y}
X^{2}Y = t(X,Y):= 35\cdot 1 -22X-Y-46X^{2}+3XY+11Y^{2},
\end{equation}
and
\begin{equation}\label{Y3}
Y^{3}= q(X,Y):= d_{1}\cdot 1+ d_{2} X + d_{3}Y+ d_{4}X^{2}+
d_{5}XY+d_{6}Y^{2}+d_{7}XY^{2},
\end{equation}
where
$d_{1}= \frac{3(487658-1651r)}{1447}$,
$d_{2}= \frac{3(-342075+1157r)}{1447}$,
$d_{3}= \frac{2(-2131598+6591r)}{18811}$,
$d_{4}= \frac{-2000094+6773r}{1447}$,
$d_{5}= \frac{2338519-6591r}{18811}$,
$d_{6}= \frac{2(-316575+1079r)}{1447}$,
$d_{7}= \frac{-48015+169r}{1447}$.
Thus, $M_{3}$ satisfies all of the hypotheses of Theorem \ref{main},
except that (\ref{X2Y}) is an ``extra" dependence relation
(not a linear combination of the relations defined in (\ref{X3})
and (\ref{Y3}). \ We claim that $M_{3}$ does not admit a
moment matrix extension block $B(4)$ such that 
$\bpm
M_{3} & B(4)
\epm $
is recursively generated.
Indeed, if such a block existed, then in the column space of
$\bpm
M_{3} & B(4)
\epm $
we would have
$X^{3}Y = (yp)(XY) :=40Y -24XY+4Y^{2}-53X^{2}Y-2XY^{2}+13Y^{3}$
and also
$X^{3}Y = (xt)(X,Y):= 35X -22X^{2}-XY-46X^{3}+3X^{2}Y+11XY^{2}$.
A calculation shows that 
$\langle (yp)(X,Y)-(xt)(X,Y),XY^{2}\rangle=
\frac{-49462+169r}{13}$, so for $r\not = \frac{49462}{169}$,
$X^{3}Y$ is not well-defined. \ Thus, the conclusions of
Theorem \ref{main} do not hold for $M_{3}$ (and thus there is no representing measure). \qed

\end{example}

By contrast with the preceding example, we next show
that if $M_{d} \in RD$, with all column dependence relations
of strictly lower degree, then $M_{d}$ does admit an $RG$
extension.
\begin{theorem} \label{RDnew}
Suppose $M_{d}$ is positive and recursively generated,
and satisfies (\ref{Xrec})-(\ref{Yrec}). \ If each column relation  in $M_{d}$  can be expressed as
$X^{i}Y^{j}=r(X,Y)$ with $deg~r<i+j$, then $M_{d}$
admits a unique $RG$ extension.
\end{theorem}

We present the proof of Theorem \ref{RDnew} in Section \ref{PROOF2}. \ Finally, we note that in applying the algorithm, Theorem 
\ref{rdext} or Theorem \ref{RDnew} may apply at some extension steps, but not at others. \ Consider \cite[Example 4.15]{finitevariety},
which concerns a recursively determinate $M_{5}$ with $n=m=d=5$, $deg~p=5$, $deg~q=4$. \ The moment matrix $M_{5}$ satisfies the 
hypotheses of Theorem \ref{rdext} (with the roles of $p$ and $q$ reversed). \ The $RG$ extension $M_{6}$ is positive semidefinite and
satisfies the hypotheses of Theorem \ref{rdext}. \ The $RG$ extension $M_{7}$ is also positive semidefinite, but has a new column 
relation, $X^3Y^4=r(X,Y) \; (deg~r=6)$, that is not recursively determined from $X^5=p(X,Y)$ or $Y^5=q(X,Y)$. \ Thus, Theorem \ref{rdext}
does not apply to $M_{7}$, nor does Theorem \ref{RDnew} (since $deg~p=5=n$). \ Nevertheless, in this case, when the algorithm is applied 
to $M_{7}$, a flat extension $M_{8}$ (and a measure) results.    


\section{An extension sequence that fails at the second stage}\label{Sect3}

\setcounter{equation}{0}

Recall that
in the most important case of recursive determinacy,
a positive, flat $M_{d}$ admits unique positive,
recursively generated extensions of all orders, $M_{d+1},\ldots
M_{d+k},\ldots$, leading to a unique representing measure. \
Further, in all of the examples of \cite{tcmp3}, \cite{tcmp11} and \cite{finitevariety}, when a positive,
recursively generated,
recursively
determinate $M_{d}$ fails to have a representing measure, it is
because it fails to admit a positive, recursively generated
extension $M_{d+1}$. \ These results suggest the question as to  whether
a positive, recursively generated, recursively determinate $M_{d}$
which admits a positive, recursively generated $M_{d+1}$ necessarily
admits positive,
recursively generated extensions of all orders
(and thus a representing measure) \cite[Question 4.19]{finitevariety}. \
In this section we provide a negative answer to this question. \ 
In the sequel we construct a positive, recursively
generated, recursively determinate
$M_{4}(\beta^{(8)})$ which admits a positive, recursively generated extension
$M_{5}$, but
 such that $M_{5}$
  fails to admit a positive, recursively
generated extension $M_{6}$. \ It then follows
from the Bayer-Teichmann Theorem that $\beta^{(8)}$ has no
representing measure.

We define $M_{4}$ by defining its component blocks in the
decomposition
\begin{equation}\label{m4blocks}
M_{4} =
\bpm
M_{3} & B(4)   \\
B(4)^{T} & C(4) 
\epm.
\end{equation}

We begin by setting $\beta_{00}= \beta_{20}=\beta_{02}= \beta_{22} = 1$,
$\beta_{40}=\beta_{04} = \beta_{42}=\beta_{24}=2$, $\beta_{60}=
\beta_{06} = 5$, and all other moments up to degree 6 set to $0$, so that
\begin{equation}\label{m3}
M_{3} =
\bpm
1  &  0 & 0 & 1 & 0 & 1 & 0 & 0 & 0 & 0   \\
0  &  1 & 0 & 0 & 0 & 0 & 2 & 0 & 1 & 0   \\
0  &  0 & 1 & 0 & 0 & 0 & 0 & 1 & 0 & 2   \\
1  &  0 & 0 & 2 & 0 & 1 & 0 & 0 & 0 & 0   \\
0  &  0 & 0 & 0 & 1 & 0 & 0 & 0 & 0 & 0   \\
1  &  0 & 0 & 1 & 0 & 2 & 0 & 0 & 0 & 0   \\
0  &  2 & 0 & 0 & 0 & 0 & 5 & 0 & 2 & 0   \\
0  &  0 & 1 & 0 & 0 & 0 & 0 & 2 & 0 & 2   \\
0  &  1 & 0 & 0 & 0 & 0 & 2 & 0 & 2 & 0   \\
0  &  0 & 2 & 0 & 0 & 0 & 0 & 2 & 0 & 5   
\epm.
\end{equation}

We next set

\begin{equation}\label{B4}
B(4) =
\bpm
2  &  0 & 1  & 0 & 2    \\
0  &  0 & 0 & 0 & 0   \\
0  &  0 & 0 & 0 & 0   \\
5  &  0 & 2 & 0 & 2   \\
0  &  2 & 0 & 2 & 0   \\
2  &  0 & 2 & 0 & 5   \\
a  &  b & 0 & 0 & 0   \\
b  &  0 & 0 & 0 & 0    \\
0  &  0 & 0 & 0 & g   \\
0  &  0 & 0 & g & h  
\epm,
\end{equation}
where $\beta_{70}=a$, $\beta_{61}=b$, $\beta_{16}=g$, $\beta_{07}=h$,
and all other degree 7 moments equal 0.
Let
\begin{equation}\label{p}
p(x,y):= ax^{3} + b x^{2}y+ 3x^{2}-by-2ax-1
\end{equation}
and
\begin{equation}\label{q}
q(x,y):= gxy^{2}+hy^{3}+ 3y^{2}-2hy-gx-1,
\end{equation}
so that in the column space of
$ \bpm
M_{3} & B(4)  
 \epm $, we have the relations
 \begin{equation}\label{x4}
X^{4}= p(X,Y)
\end{equation}
and
\begin{equation}\label{y4}
Y^{4}=q(X,Y),
\end{equation}
and $\text{rank}~
 \bpm
M_{3} & B(4)  
 \epm = 13$.

We complete the definition of a
recursively determinate $M_{4}$
by extending the relations
(\ref{x4}) and (\ref{y4})
to the columns of 
$
\bpm
B(4)^{T} & C(4) 
\epm,$ leading to
\begin{equation}\label{C4}
C(4) =
\bpm
13+a^{2}+b^{2}  &  ab & 5  & 0 & 4    \\
ab  &  5 & 0 & 4 & 0   \\
5  &  0 & 4 & 0 & 5   \\
0  &  4 & 0 & 5 & gh   \\
4  &  0 & 5 & gh & 13+g^{2}+h^{2}   
\epm.
\end{equation}
Since $M_{3}\succ 0$ (positive and invertible), we see that
$M_{4} \succeq 0$ with rank 13
if and only if
$\Delta(4):= C(4) - B(4)^{T}M_{3}^{-1}B(4)\succ  0$. \ In view of
 (\ref{x4}) and (\ref{y4}), this is equivalent to the
 positivity of the compression of $\Delta(4)$ to rows and
 columns indexed by $X^{3}Y$, $X^{2}Y^{2}$, $XY^{3}$, i.e.,
\begin{equation}\label{Delta4}
[\Delta(4)]_{\{X^{3}Y,X^{2}Y^{2},XY^{4}\}} \equiv
\bpm
1-b^{2}  &  0 & 0  \\
0  &  1 & 0    \\
0  &  0 &   1-g^{2}
\epm \succ 0.
\end{equation}
Thus, if $b$ and $g$ satisfy
$1-b^{2}>0$ and $1-g^{2}>0$, then
$M_{4}$ is positive, recursively generated, and recursively
determinate, with $\text{rank}~M_{4} = 13$, so $M_{4}$
satisfies the hypotheses of Theorem \ref{main}.

We next seek to extend $M_{4}$ to a positive
and recursively generated $M_{5}$.
In view of (\ref{x4}) and (\ref{y4}), this can
only be accomplished by defining
\begin{equation}\label{x5}
X^{5} := (xp)(X,Y)
\end{equation}
and
\begin{equation}\label{y5}
Y^{5} := (yq)(X,Y).
\end{equation}
Theorem \ref{main} implies
that
the resulting
$B(5)$ is well-defined and
satisfies $Ran~B(5)\subseteq Ran~M_{4}$,
so there exists $W$ satisfying
$B(5) = M_{4}W$. \ A calculation now shows
that if we define $C(5)$ via
(\ref{x5}) and (\ref{y5})
(as we must to preserve recursiveness),
 then
$M_{5}\succeq 0$ if and only if
$$\Delta(5)\equiv C(5) -B(5)^{T}W =
\bpm
0 &  0 & 0  & 0 & 0 & 0   \\
0  &  0 & 0 & 0 & 0 & 0  \\
0  &  0 &  \frac{-1+2b^{2}}{-1+b^{2}} & bg & 0 & 0   \\
0  &  0 & bg &  \frac{-1+2g^{2}}{-1+g^{2}} & 0 & 0   \\
0  &  0 & 0 & 0 & 0 & 0 \\  
0  &  0 & 0 & 0 & 0 & 0  
\epm \succeq 0.
$$
Thus,
using
nested determinants, and since  $b^{2}<1$, we see that
 $M_{5}$
 is positive and recursively generated, with
 $\text{rank}~M_{5} = 15$ 
 if and only if
\begin{equation}\label{btest}
b^{2} < \frac{1}{2}
\end{equation}
and
\begin{equation}\label{bgtest}
1-2b^{2}-2g^{2}+3b^{2}g^{2}+b^{4}g^{2}+b^{2}g^{4}-b^{4}g^{4} > 0.
\end{equation}
For example, setting $b=g= \frac{1}{4}$,
the expression in (\ref{bgtest}) equals $\frac{49951}{65536} (>0)$,
so it follows that $M_{5}$ is positive and recursively
generated, with $\text{rank}~M_{5} = 15$, whence $M_{5}$ satisfies
the conditions of Theorem \ref{main}.

With these values for $b$ and $g$ (or using other
appropriate values), we next attempt to define
a positive and recursively generated extension $M_{6}$.
This can only be done by defining $X^{6}:= (x^{2}p)(X,Y)$
and $Y^{6}:= (y^{2}q)(X,Y)$.
Theorem \ref{main} implies that the resulting $B(6)$ is well-defined
and that there is a matrix $V$ such that $B(6) = M_{5}V$.
Further, $C(6)$ is uniquely defined via the
preceding column relations. \ $M_{6}$ as thus defined is
recursively generated (by construction), but we will show that
it need not be positive. \
Indeed, a calculation shows that  $\Delta(6)
\equiv C(6)- B(6)^{T}V$ 
is identically 0 except perhaps for the element
in the row and column indexed by
$X^{3}Y^{3}$ (the row 4, column 4 element), which  is equal to
$$\frac{(1-3b^{2}+b^{4}-ab^{2}g+ab^{4}g+bh-2b^{3}h)
(-1-ag+3g^{2}+2ag^{3}-g^{4}+bg^{2}h-bg^{4}h)}
{-1+2b^{2}+2g^{2}-3b^{2}g^{2}-b^{4}g^{2}-b^{2}g^{4}+b^{4}g^{4}}.$$
Note that the denominator of the preceding expression is the
negative of the expression in (\ref{bgtest}),
and is thus strictly negative.
Thus $M_{6}$ is positive if and only if
\begin{equation}\label{m6test}
\eta :=(1-3b^{2}+b^{4}-ab^{2}g+ab^{4}g+bh-2b^{3}h)
(-1-ag+3g^{2}=2ag^{3}-g^{4}+bg^{2}h-bg^{4}h)\le 0.
\end{equation}
With $b=g=\frac{1}{4}$, we have
$$\eta = \frac{(-836+15a-224 h)(836+224a-15h)}{1048576}.$$
If we choose $a$ and $h$ so that $\eta=0$, then $M_{6}$ is a
flat extension of $M_{5}$, and $\beta \equiv \beta^{(8)}$ has
a $15$-atomic representing measure.
If we choose $a$ and $h$ so that $\eta <0$, then $M_{6}$ is
positive with rank 16, and since, in Corollary \ref{flatext}, $n=m=4$ and
$d=6$, it follows that $M_{6}$ has a flat extension $M_{7}$.
However, if we choose $a$ and $h$ so that $\eta>0$ 
(e.g., with
 $h=0$ and $a>\frac{836}{15}$), then
$M_{6}$ is not positive, whence $\beta$ has no representing measure.

\section{Proof of Theorem \ref{main}} \label{PROOF}

\setcounter{equation}{0}

The proof of Theorem \ref{main} entails two mains steps: (i) the construction of the block $B(d+1)$ from the column relations (\ref{X}) and 
(\ref{Y}) so that $(M_d \; B(d+1))$ is recursively generated; and (ii) the verification that $Ran~B(d+1) \subseteq Ran~M_d$. \ 

{\bf STEP (i)}: \ Step (i) will follow from a series of five auxiliary results (Lemmas \ref{old=new} - \ref{righthankel}). \ 
To begin the formal definition of $B(d+1)$, note that blocks $B[0,d+1],\ldots,B[d-1,d+1]$ are completely defined in terms
of moments in $M_{d}$. \ Indeed, for $0\le i\le d+1$, $0\le j\le d-1$, and $h,k\ge 0$ with $h+k=j$, the component
of $B[j,d+1]$ in row $X^{h}Y^{k}$ and column $X^{i}Y^{d+1-i}$, which we denote by $\langle X^{i}Y^{d+1-i},X^{h}Y^{k}\rangle$,
must equal $\beta_{i+h,d+1-i+k}$. \ Note also that for $i\ge n$,
the above component is alternately defined by (\ref{Xton_rec}),
so we must show that the two definitions agree.

\begin{lemma}\label{old=new}
For $0\le f \le d+1-n$ and $i,j\ge 0$ with $i+j\le d-1$,
the entry in column $X^{n+f}Y^{d+1-n-f}$, row $X^{i}Y^{j}$, as defined
by   (\ref{Xton_rec}),  coincides with the moment inherited
from $M_{d}$ by moment matrix structure, $\beta_{n+f+i,d-n-f+j+1}$.
\end{lemma}

\begin{proof}  Consider first the case when $d-n-f\ge 0$.
From (\ref{Xton_rec}), we have
\begin{eqnarray*}
X^{n+f}Y^{d+1-n-f} &:=& (x^{f}y^{d+1-n-f}p)(X,Y) \\
&\equiv& \sum \limits_{r,s\ge 0, r+s\le n-1}^{} a_{rs}X^{r+f}Y^{s+d+1-n-f} \quad (0\le f\le d+1-n),
\end{eqnarray*}
so
\begin{eqnarray*}
\langle X^{n+f}Y^{d+1-n-f},X^{i}Y^{j}\rangle =\sum  a_{rs}\langle X^{r+f}Y^{s+d+1-n-f},  X^{i}Y^{j}\rangle.
\end{eqnarray*}
Since $r+f+s+d+1-n-f\le d$, $s+d+1-n-f\ge 1$ and $i+j\le d-1$,
using the moment matrix structure of the blocks of $M_{d}$ we may express
the last sum as
$$
\sum a_{rs}\langle X^{r+f}Y^{s+d-n-f},  X^{i}Y^{j+1}\rangle . 
$$
Now (\ref{Xrec}) implies that in $M_{d}$ the later expression is equal to 
$$
\langle X^{n+f}Y^{d-n-f},X^{i}Y^{j+1}\rangle
= \beta_{n+f+i,d-n-f+j+1}.
$$
For the case $f=d-n+1$ and $i+j\le d-1$,
\begin{eqnarray*}
\langle X^{d+1},X^{i}Y^{j}\rangle &=& \sum    a_{rs}\langle X^{r}Y^{s}X^{d+1-n},  X^{i}Y^{j}\rangle \\
&=&\sum  a_{rs}\langle X^{r}Y^{s}X^{d-n}, X^{i+1}Y^{j}\rangle \\
&=&\langle X^{d},X^{i+1}Y^{j}\rangle \\
&=& \beta_{d+i+1,j}.
\end{eqnarray*}
\end{proof}

We have just verified that in the left recursive band, 
in blocks of degree at most $d-1$, each column element
coincides with the corresponding ``old" moment from $M_{d}$.
Old moments are also used to {\it{define}} the central (nonrecursive)
 band of columns in blocks of degree at most $d-1$. \ We next use these left and central bands,
together with (\ref{Y^m_rec}), to show that the column elements in the right recursive band,
in blocks of degree at most $d-1$,
also agree with corresponding old moments.

\begin{lemma}\label{oldright}
For $0\le k\le d+1-m$, $i,j\ge0$, $i+j\le d-1$,
column $X^{d+1-m-k}Y^{m+k}$, as defined by (\ref{Y^m_rec}), satisfies
$\langle X^{d+1-m-k}Y^{m+k},X^{i}Y^{j}\rangle =
\beta_{i+d+1-m-k,m+k+j}$.
\end{lemma}
\begin{proof}
The proof is by induction on $k$. \ For $k=0$, we show that
$\langle X^{d+1-m}Y^{m},X^{i}Y^{j}\rangle =
\beta_{i+d+1-m,m+j}$. \
From (\ref{Y^m_rec}), we have
\begin{eqnarray*}
\langle X^{d+1-m}Y^{m},X^{i}Y^{j}\rangle &=& \sum \limits_{a,b\ge 0, a+b\le m-1} 
\alpha_{ab} \langle X^{d+1-m+a}Y^{b},X^{i}Y^{j}\rangle \\
& \quad & + \sum \limits_{c,e\ge 0, c+e=m, e<m} 
\gamma_{ce} \langle X^{c+d+1-m}Y^{e},X^{i}Y^{j}\rangle.
\end{eqnarray*}
Since $a+b<m$, then in $M_{d}$,
$$ \langle X^{d+1-m+a}Y^{b},X^{i}Y^{j}\rangle =
\beta_{d+1-m+a+i,b+j}.$$
Since $e<m$, $ \langle X^{c+d+1-m}Y^{e},X^{i}Y^{j}\rangle$
is in either the left or central band, and thus equals the old moment
$\beta_{c+d+1-m+i,e+j}$.
Now
 
$$
\langle X^{d+1-m}Y^{m},X^{i}Y^{j}\rangle =
\displaystyle \sum \limits_{a,b\ge 0, a+b\le m-1} 
\alpha_{ab} \beta_{d+1-m+a+i,b+j}+
\displaystyle \sum \limits_{c,e\ge 0, c+e=m, e<m} 
\gamma_{ce} \beta_{c+d+1-m+i,e+j}.
$$

In $M_{d}$, the latter expression equals
\begin{eqnarray*}
\sum \alpha_{ab} \langle X^{d-m+a}Y^{b},X^{i+1}Y^{j}\rangle + \sum \gamma_{ce} \langle X^{c+d-m}Y^{e},X^{i+1}Y^{j}\rangle &=& \langle X^{d-m}Y^{m},X^{i+1}Y^{j}\rangle \\
&=& \beta_{d-m+i+1,m+j},
\end{eqnarray*}

as desired. \ We next assume the result is true for $0,\ldots,k-1$. \ Consider first
the case when $k<d+1-m$.
We have
\begin{eqnarray*}
\langle X^{d+1-m-k}Y^{m+k},X^{i}Y^{j}\rangle &=&
\sum \limits_{a,b\ge 0, a+b\le m-1} 
\alpha_{ab} \langle X^{d+1-m-k+a}Y^{b+k},X^{i}Y^{j}\rangle \\
& \quad & + \sum \limits_{c,e\ge 0, c+e=m, e<m} 
\gamma_{ce} \langle X^{c+d+1-m-k}Y^{e+k},X^{i}Y^{j}\rangle.
\end{eqnarray*}
The term
$\langle X^{d+1-m-k+a}Y^{b+k},X^{i}Y^{j}\rangle$
is a component of $M_{d}$, and thus equals the corresponding moment.
Since $e+k\le m+ (k-1)$, $  X^{c+d+1-m-k}Y^{e+k}$ is, by induction,
a column for which the elements of row-degree $i+j$ are old moments.
Thus, 
\begin{eqnarray*}
\langle X^{d+1-m-k}Y^{m+k},X^{i}Y^{j}\rangle = \sum \alpha_{ab} \beta_{d+1-m-k+a+i,b+k+j}+ \sum \gamma_{ce} \beta_{c+d+1-m-k+i,e+k+j}.
\end{eqnarray*}
In $M_{d}$, the last expression equals
\begin{eqnarray*}
\sum \alpha_{ab} \langle X^{d+a-m-k}Y^{b+k},X^{i+1}Y^{j}\rangle \quad \quad \\ 
\quad \quad + \sum \gamma_{ce} \langle X^{c+d-m-k}Y^{e+k},X^{i+1}Y^{j}\rangle 
&=& \langle X^{d-m-k}Y^{m+k},X^{i+1}Y^{j}\rangle \\
&=& \beta_{d-m-k+i+1,m+k+j}.
\end{eqnarray*}

Finally, we consider the case $k=d+1-m$.
We have
\begin{eqnarray*}
\langle Y^{d+1},X^{i}Y^{j}\rangle &=& \sum \limits_{a,b\ge 0, a+b\le m-1} 
\alpha_{ab} \langle X^{a}Y^{b+d+1-m},X^{i}Y^{j}\rangle \\
&& \quad + \sum \limits_{c,e\ge 0, c+e=m, e<m} 
\gamma_{ce} \langle X^{c}Y^{e+d+1-m},X^{i}Y^{j}\rangle.
\end{eqnarray*}
Since $e<m$, then $c\ge1$, so $X^{c}Y^{e+d+1-m}$ is to the
left of $Y^{d+1}$, i.e., $c=d+1-m-k^{\prime}$ for
$k^{\prime} = d+1-m-c<k$. \ Thus, by induction,
\begin{eqnarray*}
\langle Y^{d+1},X^{i}Y^{j}\rangle = \sum \alpha_{ab} \beta_{a+i,b+d+1-m+j}+ \sum \gamma_{ce} \beta_{c+i,e+d+1-m+j}.
\end{eqnarray*}
In $M_{d}$, the last expression equals
\begin{eqnarray*}
\sum \alpha_{ab} \langle X^{a}Y^{b+d-m},X^{i}Y^{j+1}\rangle + \sum \gamma_{ce} \langle X^{c}Y^{e+d-m},X^{i}Y^{j+1}\rangle &=&
 \langle Y^{d},X^{i}Y^{j+1}\rangle \\
 &=& \beta_{i,d+j+1},
\end{eqnarray*}
as desired.
\end{proof}

To complete the definition of $B(d+1)$ we must define $B[d,d+1]$.
Within this proposed block, we first use (\ref{Xton_rec}) to define the left
recursive band, $X^{d+1},\ldots,X^{n}Y^{d+1-n}$.
Note that between the end of the left band, $X^{n}Y^{d+1-n}$,
and the beginning of the right band, $X^{n+1-m}Y^{m}$,
there is a central band of $n+m-d-2$ columns; set $\delta := n+m-d-1$.
In row $X^{d}$, each of the components in  the central columns,
$\langle X^{n-1}Y^{d+2-n},X^{d}\rangle,\ldots,\langle X^{d+2-m}Y^{m-1},X^{d}\rangle$, corresponds via a cross-diagonal to a
component of column $X^{n}Y^{d+1-n}$ (whose value is known from (\ref{Xton_rec})),
i.e., 
$$
\langle X^{n-j}Y^{d+1-n+j},X^{d}\rangle
=\langle X^{n}Y^{d+1-n},X^{d-j}Y^{j}\rangle \ (1\le j\le m+n-d-2).
$$
We may thus use (\ref{Y^m_rec}) to define
$\langle X^{d+1-m}Y^{m},X^{d}\rangle$, and we extend
the latter value along the central-band section of the cross-diagonal to which it belongs. \
Next, in row $X^{d-1}Y$, we  use this value with (\ref{Y^m_rec})
to define $\langle X^{d+1-m}Y^{m},X^{d-1}Y\rangle$, and
we extend this value along 
the central-band section of
its cross-diagonal.
Proceeding in this way, we completely define column $X^{d+1-m}Y^{m}$
and insure that it is Hankel with respect to the central band.
Finally, we use (\ref{Y^m_rec}) to define column $X^{d-m}Y^{m+1}$, and, successively,
 $X^{d-m-1}Y^{m+2},\ldots,Y^{d+1}$. \ This completes the definition
of a proposed block $B[d,d+1]$. \ However, to ensure that it is
well-defined as a moment block, we must check that for a cross-diagonal which intersects
columns $X^{n}Y^{d+1-n}$ and $X^{d+1-m}Y^{m}$, the components of the
cross-diagonal in these columns agree in value, i.e., the values arising
from (\ref{Xton_rec}) are consistent with those arising from (\ref{Y^m_rec}). \ More generally,
we need to show that the block we have defined is constant on cross-diagonals.

To show that $B[d,d+1]$ is well-defined and Hankel, we begin
with the
 following general result concerning
adjacent columns that are recursively determined from the same
column dependence relation.
Suppose in $Col~M_{d}$ there is a dependence relation
$X^{c}Y^{e} = p(X,Y)$, where $c+e=d$ and
$p(x,y) \equiv \displaystyle \sum \limits_{a,b\ge 0, a+b\le d-1}^{}
\alpha_{ab}x^{a}y^{b}\in \mathcal{P}_{d-1}$.
Then the elements of  $Col~M_{d}$ defined by
\begin{eqnarray*}
X^{c+1}Y^{e} \equiv (xp)(X,Y) := \sum \limits_{a,b\ge 0, a+b\le d-1}^{}\alpha_{ab}X^{a+1}Y^{b}
\end{eqnarray*}
and
\begin{eqnarray*}
X^{c}Y^{e+1} \equiv (yp)(X,Y) := \sum \limits_{a,b\ge 0, a+b\le d-1}^{}\alpha_{ab}X^{a}Y^{b+1}
\end{eqnarray*}
are Hankel with respect to each other,
as follows.
\begin{lemma}\label{hankel}
For $i,j\ge 0$, $i+j\le d$, $j>0$,
$$\langle X^{c+1}Y^{e},X^{i}Y^{j} \rangle
= \langle X^{c}Y^{e+1},X^{i+1}Y^{j-1} \rangle$$
\end{lemma}
\begin{proof}
We have
$$\langle X^{c+1}Y^{e},X^{i}Y^{j} \rangle
 =
 \displaystyle \sum \limits_{a,b\ge 0, a+b\le d-1}^{}
\alpha_{ab} \langle X^{a+1}Y^{b}, X^{i}Y^{j} \rangle,$$
and since each row and column in  the last sum
has degree at most $d$, relative to $M_{d}$ we may rewrite this sum as
$$ \displaystyle \sum \limits_{a,b\ge 0, a+b\le d-1}^{}
\alpha_{ab} \langle X^{a}Y^{b+1}, X^{i+1}Y^{j-1} \rangle
= \langle X^{c}Y^{e+1},X^{i+1}Y^{j-1} \rangle.$$
This completes the proof.
 \end{proof}

It follows immediately from Lemma \ref{hankel} that the left 
recursive band in $B[d,d+1]$ is constant on cross-diagonals.
 We next check that if an element of a column in
 the non-recursive central band can
be reached on a cross-diagonal which intersects both columns
$X^{n}Y^{d+1-n}$ (at the edge of the left recursive band)
and $X^{d+1-m}Y^{m}$ (at the edge of the right recursive band),
then the values obtained from
both of these columns agree. \ This is the substance of the following
lemma.
\begin{lemma}\label{Hankellemma} For $0\le k\le 2d+1-n-m$,
\begin{equation}\label{hankelformula}       \langle X^{d+1-m}Y^{m},X^{d-k}Y^{k} \rangle 
= \langle X^{n}Y^{d+1-n},X^{d-\delta-k}Y^{\delta + k} \rangle.
\end{equation}
\end{lemma}

\begin{proof}
The proof is by induction on $k$. \ We begin with the base case, $k=0$, and seek to show that
$\langle X^{d+1-m}Y^{m},X^{d} \rangle
= \langle X^{n}Y^{d+1-n},X^{d-\delta}Y^{\delta} \rangle $ \ (recall that $\delta:=n+m-d-1$). \
Using (\ref{Y^m_rec}), we  may express
$ \langle X^{d+1-m}Y^{m},X^{d} \rangle$ as
\begin{equation}\label{rightbase}
 \displaystyle \sum \limits_{a,b\ge 0, a+b\le m-1} 
\alpha_{ab} \langle X^{d+1-m+a}Y^{b},X^{d}\rangle +
\displaystyle \sum \limits_{c,e\ge 0, c+e=m, e<m} 
\gamma_{ce} \langle X^{d+1-e}Y^{e},X^{d}\rangle.
\end{equation}
Note that $ \langle X^{d+1-m+a}Y^{b},X^{d}\rangle$ is a
component of $M_{d}$; further, since $e<m$,
$ \langle X^{c+d+1-m}Y^{e},X^{d}\rangle$ is the
endpoint of a cross-diagonal that lies entirely in the
left and central bands, and is thus constant. \ Therefore, we may
rewrite  (\ref{rightbase}) as
$$
\displaystyle \sum \limits_{a,b\ge 0, a+b\le m-1} 
\alpha_{ab} \langle X^{d}, X^{d+1-m+a}Y^{b}\rangle +
\displaystyle \sum \limits_{c,e\ge 0, c+e=m, e<m} 
\gamma_{ce} \langle  X^{d+1},    X^{d-e}Y^{e} \rangle
$$
\begin{eqnarray*}
&=& \sum \alpha_{ab} \langle  \sum a_{rs}  X^{d-n+r}Y^{s}, X^{d+1-m+a}Y^{b}\rangle +
\sum \gamma_{ce} \langle  \sum a_{rs} X^{d-n+r+1}Y^{s}, X^{d-e}Y^{e} \rangle \\
&=& \sum a_{rs}  \sum \alpha_{ab} \langle  X^{d-n+r}Y^{s}, X^{d+1-m+a}Y^{b}\rangle +
\sum a_{rs} \sum \gamma_{ce} \langle   X^{d-n+r+1}Y^{s}, X^{d-e}Y^{e} \rangle \\
&=& \sum a_{rs} ( \sum \alpha_{ab} \langle  X^{d+1-m+a}Y^{b}, X^{d-n+r}Y^{s}\rangle +
\sum \gamma_{ce} \langle    X^{d-e}Y^{e}, X^{d-n+r+1}Y^{s} \rangle) \\
&=& \sum a_{rs}    \langle  \sum \alpha_{ab}  X^{d-m+a}Y^{b} +
\sum \gamma_{ce}   X^{d-e}Y^{e},   X^{d-n+r+1}Y^{s} \rangle \\
&=& \sum a_{rs}    \langle    X^{d-m}Y^{m} , X^{d-n+r+1}Y^{s} \rangle.
\end{eqnarray*}
Since $\delta=m+n-d-1$, in $M_{d}$ the last sum is equal to
$$
\sum a_{rs}    \langle    X^{r}Y^{d+1-n+s}, X^{d-\delta}Y^{\delta} \rangle = \langle X^{n}Y^{d+1-n},X^{d-\delta}Y^{\delta} \rangle,
$$
which completes the proof of the base case. 

We assume now that (\ref{hankelformula}) holds for $0,\ldots, k-1$, with  $k-1 < 2d+1-n-m$.
To establish (\ref{hankelformula}) for $k$, we consider first the case $d-k \ge n$.
Let us write $\kappa:=\langle X^{d+1-m}Y^{m},X^{d-k}Y^{k} \rangle$ as
\begin{eqnarray} \label{levelk}
\kappa&=& \sum \limits_{a,b\ge 0, a+b\le m-1} \alpha_{ab} \langle X^{d+1-m+a}Y^{b},X^{d-k}Y^{k}\rangle \nonumber \\
& \quad & + \sum \limits_{c,e\ge 0, c+e=m, e<m,d+1-e\ge n} \gamma_{ce} \langle X^{d+1-e}Y^{e},X^{d-k}Y^{k}\rangle \nonumber \\ 
& \quad & + \sum \limits_{c,e\ge 0, c+e=m, e<m,d+1-e\ge n,d+1-e<n} \gamma_{ce} \langle X^{d+1-e}Y^{e},X^{d-k}Y^{k}\rangle.
\end{eqnarray}
Note that the components in the first sum of (\ref{levelk}) lie in $M_{d}$. \ In the third sum,
since $d+1-e<n$, column $X^{d+1-e}Y^{e}$ is in the middle band, and the component
$\gamma:=\langle X^{d+1-e}Y^{e},X^{d-k}Y^{k}\rangle$ 
lies on a cross-diagonal $\sigma$ strictly above the cross-diagonal for $\kappa$. \ Either because $\sigma$ does not
intersect column $X^{d+1-m}Y^{m}$, or by induction if it does, we see that $\gamma$ has the same value as 
$\langle X^{n}Y^{d+1-n},X^{d-k-(n-(d+1-e))}Y^{k+n-(d+1-e)}\rangle$
(on the same cross-diagonal).
Thus (\ref{levelk}) can be expressed as
\begin{eqnarray} \label{levelk_2}
\kappa &=& \sum \limits_{a,b\ge 0, a+b\le m-1} \alpha_{ab} \langle  X^{d-k}Y^{k}, X^{d+1-m+a}Y^{b}\rangle \nonumber \\ 
& \quad & + \sum \limits_{c,e\ge 0, c+e=m, e<m,d+1-e\ge n} \gamma_{ce} \langle X^{n}X^{d+1-e-n}Y^{e},X^{d-k}Y^{k}\rangle \nonumber \\ 
& \quad & + \sum \limits_{c,e\ge 0, c+e=m, e<m,d+1-e\ge n,d+1-e<n} \gamma_{ce} \langle X^{n}Y^{d+1-n},X^{d-k-(n-(d+1-e))}Y^{k+n-(d+1-e)}\rangle \nonumber \\ 
\medskip \medskip
&=& \sum \alpha_{ab} \langle X^{n} X^{d-k-n}Y^{k}, X^{d+1-m+a}Y^{b}\rangle \nonumber \\
& \quad & + \sum \gamma_{ce} \langle X^{n}X^{d+1-e-n}Y^{e},X^{d-k}Y^{k}\rangle \nonumber \\ 
& \quad & + \sum \gamma_{ce} \langle  X^{n}Y^{d+1-n},X^{d-k-(n-(d+1-e))}Y^{k+n-(d+1-e)}\rangle \nonumber \\
\medskip \medskip
&=& \sum \alpha_{ab}  \sum a_{rs} \langle  X^{r+d-k-n}Y^{s+k}, X^{d+1-m+a}Y^{b}\rangle \nonumber \\ 
& \quad & + \sum \gamma_{ce} \sum a_{rs} \langle X^{r+d+1-e-n}Y^{s+e},X^{d-k}Y^{k}\rangle \nonumber \\
& \quad & + \sum \gamma_{ce}  \sum a_{rs} \langle  X^{r}Y^{s+d+1-n},X^{d-k-(n-(d+1-e))}Y^{k+n-(d+1-e)}\rangle \nonumber \\
\medskip \medskip
&=& \sum a_{rs}(\sum \alpha_{ab}  \langle  X^{r+d-k-n}Y^{s+k}, X^{d+1-m+a}Y^{b}\rangle \nonumber \\
& \quad & + \sum \gamma_{ce}  \langle X^{r+d+1-e-n}Y^{s+e},X^{d-k}Y^{k}\rangle \nonumber \\
& \quad & + \sum \gamma_{ce}  \langle  X^{r}Y^{s+d+1-n},X^{d-k-(n-(d+1-e))}Y^{k+n-(d+1-e)}\rangle ).
\end{eqnarray}
Using the symmetry of $M_{d}$ in the first and third inner sums of the last
expression, we may rewrite this
expression as
\begin{eqnarray}\label{level3}
& & \sum a_{rs}(\sum \alpha_{ab}  \langle X^{d+1-m+a}Y^{b}, X^{r+d-k-n}Y^{s+k}\rangle + \sum \gamma_{ce}  \langle X^{r+d+1-e-n}Y^{s+e},X^{d-k}Y^{k}\rangle \nonumber \\ 
& \quad & + \sum \gamma_{ce}  \langle  X^{d-k-(n-(d+1-e))}Y^{k+n-(d+1-e)}, X^{r}Y^{s+d+1-n} \rangle).
\end{eqnarray}
In the second inner sum of (\ref{level3}), 
$\langle X^{r+d+1-e-n}Y^{s+e},X^{d-k}Y^{k}\rangle$
is a component of $M_{d}$ and thus equals the moment 
$\beta_{r+d+1-e-n+d-k,s+e+k}$. \ Since $X^{d-k+r-n}Y^{s+k}$
is a row of degree at most $d-1$,  this moment
coincides with 
$\langle X^{c+d+1-m}Y^{e},X^{d-k+r-n}Y^{s+k}\rangle$
from the left band of $B[d+r+s-n,n+1]$.
Further, in the third inner sum of (\ref{level3}),
$$
\langle  X^{d-k-(n-(d+1-e))}Y^{k+n-(d+1-e)}, X^{r}Y^{s+d+1-n} \rangle
$$
is also a component of $M_{d}$, equal to $\beta_{r+d-k-(n-(d+1-e)),s+k+n-(d+1-e)+d+1-n}$, and this moment coincides with $\langle X^{d+1-e}Y^{e},X^{r+d-k-n}Y^{s+k}\rangle$ from the middle band 
in $B[d+r+s-n,d+1]$. \ Thus, the expression in (\ref{level3}) can be written as
\begin{eqnarray}\label{level4}
&& \sum a_{rs}(\sum \limits_{a,b\ge 0, a+b\le m-1} \alpha_{ab} \langle X^{d+1-m+a}Y^{b}, X^{r+d-k-n}Y^{s+k}\rangle \nonumber \\
& \quad & + \sum  \limits_{c,e\ge 0, c+e=m, e<m,d+1-e\ge n} \gamma_{ce} \langle X^{d+1-m+c}Y^{e},X^{r+d-k-n}Y^{s+k}\rangle \nonumber \\ 
& \quad & + \sum  \sum  \limits_{c,e\ge 0, c+e=m, e<m,d+1-e < n} \gamma_{ce}  \langle  X^{d+1-m+c}Y^{e}, X^{r+d-k-n}Y^{s+k} \rangle),
\end{eqnarray}
which equals
\begin{equation}\label{level5}
\sum a_{rs}  \langle  X^{d+1-m}Y^{m}, X^{r+d-k-n}Y^{s+k} \rangle).
\end{equation}
Since $ X^{r+d-k-n}Y^{s+k}$
is a row of degree at most $d-1$, Lemma \ref{oldright}  implies
that the expression in (\ref{level5}) equals
\begin{eqnarray*}
\sum a_{rs} \beta_{d+1-m+r+d-k-n,m+s+k} &=& \sum a_{rs} \beta_{r+d-\delta-k,s+d+1-n+\delta+k} \\
&=& \sum a_{rs} \langle  X^{r}Y^{s+d+1-n}, X^{d-\delta-k}Y^{\delta+k} \rangle \\
&=&  \langle  X^{n}Y^{d+1-n}, X^{d-\delta-k}Y^{\delta+k} \rangle .
\end{eqnarray*}
This completes the proof of the induction step for (\ref{hankelformula}) when $d-k\ge n$.

We next treat the case when $d-k<n$, which implies $\delta + k\ge m$.
We have
\begin{eqnarray}\label{decomp}
\langle  X^{n}Y^{d+1-n}, X^{d-\delta-k}Y^{\delta+k}\rangle &=& 
\sum a_{rs} \langle  X^{r}Y^{d+1-n+s}, X^{d-\delta-k}Y^{\delta+k}\rangle \nonumber  \\
&=& \sum a_{rs} \langle X^{d-\delta-k}Y^{m}Y^{\delta+k-m}, X^{r}Y^{d+1-n+s} \rangle \nonumber \\
&=& \sum a_{rs}(\displaystyle \sum \limits_{a,b\ge 0, a+b\le m-1}\alpha_{ab}\langle X^{a+ d-\delta-k}Y^{b+\delta+k-m}, X^{r}Y^{d+1-n+s} \rangle \nonumber \\
& \quad & + \sum \limits_{c,e\ge 0, c+e=m, e<m}\gamma_{ce} \langle X^{c+d-\delta-k}Y^{e+\delta+k-m}, X^{r}Y^{d+1-n+s} \rangle). \nonumber \\
&& \quad
\end{eqnarray}
Note for future reference that all of the matrix components that appear in (\ref{decomp}) come from $M_{d}$. 

We now consider 
\begin{eqnarray}\label{d2}
\langle X^{d+1-m}Y^{m},X^{d-k}Y^{k} \rangle &=& \sum \limits_{a,b\ge 0, a+b\le m-1} \alpha_{ab} \langle X^{d+1-m+a}Y^{b},X^{d-k}Y^{k}\rangle \nonumber \\
& \quad & + \sum \limits_{c,e\ge 0, c+e=m, e<m} \gamma_{ce} \langle X^{d+1-m+c}Y^{e},X^{d-k}Y^{k}\rangle \nonumber \\
&=& \sum \alpha_{ab} \langle X^{d-k}Y^{k},X^{d+1-m+a}Y^{b}\rangle \nonumber \\
& \quad & + \sum \gamma_{ce} \langle X^{d+1-e}Y^{e},X^{d-k}Y^{k}\rangle \quad \quad
\end{eqnarray}
(using symmetry of $M_d$ in the first sum). \ Since $k-(n-(d-k))$, $d+1-m+a-n+d-k$ $(= a+d-\delta-k)$, and $b+n-(d-k)$
are all nonnegative, by applying the block-Hankel property of $M_{d}$ to the
first sum in (\ref{d2}), we may rewrite the expression in (\ref{d2}) as
\begin{eqnarray}\label{d3}
&&\sum \alpha_{ab} \langle   X^{n} Y^{k-(n-(d-k))},X^{d+1-m+a-n+d-k}Y^{b+n-(d-k)}\rangle \nonumber \\
& \quad & + \sum \gamma_{ce} \langle X^{d+1-e}Y^{e},X^{d-k}Y^{k}\rangle
\end{eqnarray}
\begin{eqnarray}\label{d4}
&=& \sum \alpha_{ab} \sum a_{rs}  \langle   X^{r} Y^{s+k-(n-(d-k))},X^{d+1-m+a-n+d-k}Y^{b+n-(d-k)}\rangle \nonumber \\
& \quad & + \sum \gamma_{ce} \langle X^{d+1-e}Y^{e},X^{d-k}Y^{k}\rangle,
\end{eqnarray}
and all of the matrix components in the first double  sum of (\ref{d4}) are from $M_{d}$.
Comparing the components in the first double sums of 
(\ref{decomp}) and (\ref{d4}), we have
\begin{eqnarray*}
\langle X^{a+ d-\delta-k}Y^{b+\delta+k-m},X^{r}Y^{d+1-n+s} \rangle &=& \beta_{a+d-\delta-k+r,b+\delta+k-m+d+1-n+s} \\
&=&\beta_{r+d+1-m+a-n+d-k,s+k-(n-(d-k))+b+n-(d-k)} \\
&=& \langle X^{r} Y^{s+k-(n-(d-k))},X^{d+1-m+a-n+d-k}Y^{b+n-(d-k)}\rangle,
\end{eqnarray*}
so the first double sums of (\ref{decomp}) and (\ref{d4}) are equal.

Let us write the rightmost sum in (\ref{d4}) as
\begin{eqnarray}\label{d5}
&&\sum \limits_{c,e\ge 0, c+e=m,e<m,d+1-e\ge n} \gamma_{ce} \langle X^{n} X^{d+1-e-n}Y^{e},X^{d-k}Y^{k}\rangle  \nonumber \\
& \quad & + \sum \limits_{c,e\ge 0, c+e=m,e<m,d+1-e < n}\gamma_{ce} \langle X^{d+1-e}Y^{e},X^{d-k}Y^{k}\rangle.
\end{eqnarray}
In the second sum of (\ref{d5}), since $d+1-e<n$, the component
 $\langle X^{d+1-e}Y^{e},X^{d-k}Y^{k}\rangle$ (from the
middle band) has the same value as the component 
$\langle X^{n}Y^{d+1-n},X^{d-k-(n-(d+1+e)}Y^{k+n-(d+1-e)}\rangle$ 
on the same cross-diagonal. \ (This is because the cross-diagonal is strictly above that for $ \langle X^{d+1-m}Y^{m},X^{d-k}Y^{k} \rangle $,
so the conclusion follows by definition or induction.) \ We may now write the expression in (\ref{d5}) as
\begin{eqnarray}\label{d6}
&& \sum \limits_{c,e\ge 0, c+e=m,e<m,d+1-e\ge n} \gamma_{ce} \sum a_{rs} \langle X^{r+d+1-e-n}Y^{s+e},X^{d-k}Y^{k}\rangle \nonumber \\
& \quad & + \sum \limits_{c,e\ge 0, e<m,c+e=m,d+1-e < n} \gamma_{ce} \langle X^{n}Y^{d+1-n},X^{d-k-(n-(d+1-e))}Y^{k+n-(d+1-e)}\rangle \nonumber  \\
&=& \sum \limits_{c,e \ge 0, e<m,c+e=m,d+1-e\ge n} \gamma_{ce} \sum a_{rs}  \langle  X^{r+d+1-e-n}Y^{s+e},X^{d-k}Y^{k} \rangle \nonumber  \\
& \quad & + \sum \limits_{c,e \ge 0, c+e=m,d+1-e < n} \gamma_{ce} \sum a_{rs} \langle X^{r}Y^{s+d+1-n},X^{d-k-(n-(d+1-e))}Y^{k+n-(d+1-e)}\rangle. \quad \quad 
\end{eqnarray}
All of the matrix components in (\ref{d6}) are from $M_{d}$, so
(\ref{d6}) can be expressed as
$$
\sum a_{rs} \sum \limits_{c+e=m} \beta_{r+d+1-e-n+d-k,s+e+k}.
$$ 
It is straightforward to check
that this double sum coincides with the second double sum in (\ref{decomp})
(whose matrix components also come entirely from $M_{d}$).
This completes the proof that the second double sums  in (\ref{decomp})
and 
(\ref{d4}) have the same value, so the expressions in
(\ref{decomp})
and (\ref{d4}) are equal, which completes the proof of the induction
when $d-k<n$. \ Thus, the induction is complete.
\end{proof}

We have shown above that in $B[d,d+1]$ the columns
$X^{d+1},\ldots,X^{d+1-m}Y^{m}$ are well-defined and
Hankel with respect to one another. \ Using (\ref{Y^m_rec}), we also
successively defined columns $X^{d-m}Y^{m+1},\ldots,Y^{d+1}$.
We next show that the columns $X^{d-m+1}Y^{m},\ldots,Y^{d+1}$
are Hankel with respect to each other, so that all of $B[d,d+1]$ has the
Hankel property.
\begin{lemma}\label{righthankel}
For $0\le s \le d+1-m$ and $i,j\ge 0$ with $i+j = d$ and $j>0$, we have 
$$ \langle X^{d+1-m-s}Y^{m+s},X^{i}Y^{j}\rangle =
 \langle X^{d-m-s}Y^{m+s+1},X^{i+1}Y^{j-1}\rangle.$$
\end{lemma}
\begin{proof}
The proof is by induction on $s\ge 0$. \ For $s=0$, we have
$ \langle X^{d+1-m}Y^{m},X^{i}Y^{j}\rangle $
\begin{equation}\label{baseeqn}
= \displaystyle \sum \limits_{a,b\ge 0, a+b\le m-1} 
\alpha_{ab} \langle X^{d+1-m+a}Y^{b},X^{i}Y^{j}\rangle +
\displaystyle \sum \limits_{c,e\ge 0, c+e=m, e<m} 
\gamma_{ce} \langle X^{d+1-e}Y^{e},X^{i}Y^{j}\rangle.
\end{equation}
In the first sum, each component is from $M_{d}$. \ In the second sum,
column $X^{d+1-e}Y^{e}$ is strictly to the left 
 of $X^{d+1-m}Y^{m}$, so it is Hankel with respect to its right successor,
$X^{d-e}Y^{e+1}$ .
We may thus rewrite the expression in (\ref{baseeqn}) as  
$$\displaystyle \sum \limits_{a,b\ge 0, a+b\le m-1} 
\alpha_{ab} \langle X^{d-m+a}Y^{b+1},X^{i+1}Y^{j-1}\rangle +
\displaystyle \sum \limits_{c,e\ge 0, c+e=m, e<m} 
\gamma_{ce} \langle X^{d-e}Y^{e+1},X^{i+1}Y^{j-1}\rangle$$
\newline
$= \langle X^{d-m}Y^{m+1},X^{i+1}Y^{j-1}\rangle$.

Assume now that the Hankel property holds through $s-1$ and consider
\begin{eqnarray}\label{level_s} 
\langle X^{d+1-m-s}Y^{m+s},X^{i}Y^{j}\rangle &=& \sum \limits_{a,b\ge 0, a+b\le m-1} \alpha_{ab} \langle X^{d+1-m+a-s}Y^{b+s},X^{i}Y^{j}\rangle  \nonumber \\ 
& \quad & + \sum \limits_{c,e\ge 0, c+e=m, e<m} \gamma_{ce} \langle X^{d+1-e-s}Y^{e+s},X^{i}Y^{j}\rangle.
\end{eqnarray}
As above, in the first sum, each component is from $M_{d}$; in the second
sum, each column  $X^{d+1-e-s}Y^{e+s}$ is to the left of 
$X^{d+1-m-s}Y^{m+s}$, so the Hankel property holds for this
column by induction. \ We may thus write the expression in (\ref{level_s})
as
\begin{eqnarray*}
\sum \limits_{a,b\ a+b\le m-1} \alpha_{ab} \langle X^{d-m+a-s}Y^{b+s+1},X^{i+1}Y^{j-1}\rangle  
&+& \sum \limits_{c+e=m, e<m} \gamma_{ce} \langle X^{d-e-s}Y^{e+s+1},X^{i+1}Y^{j-1}\rangle  \\
&=& \langle X^{d-m-s}Y^{m+s+1},X^{i+1}Y^{j-1}\rangle,
\end{eqnarray*}
which completes the proof by induction.
\end{proof}

{\bf STEP (ii)}: \ The preceding results show that under the hypotheses of
Theorem \ref{main}, there exists a unique block
$B(d+1)$ that is consistent with 
recursiveness in 
$\bpm
M_{d} &  B(d+1)  \epm$
. \ To prove
Theorem \ref{main}, we must also
 show that  $Ran~B(d+1)\subseteq Ran~M_{d}$.
The following lemma is a step toward this end; it shows
that the rows of
$\bpm
M_{d} &  B(d+1)  \epm$
of the form $X^{n+f}Y^{g}$ $(f,g\ge 0, ~n+f+g\le d)$ are
recursively determined from row $X^{n}$. 

\begin{lemma}\label{recXrows}   For $i,j\ge 0, i+j\le d+1$ and
for $f,g\ge 0,~ n+f+g\le d$,
\begin{equation}\label{Xrows}
     \langle X^{i}Y^{j},X^{n+f}Y^{g}\rangle =
\displaystyle \sum \limits_{r,s\ge 0, r+s\le n-1}^{}
   a_{rs}\langle X^{i}Y^{j},  X^{r+f}Y^{s+g}\rangle.
\end{equation}
\end{lemma}
\begin{proof}
Since $M_{d}$ is real symmetric, it follows from
(\ref{Xton_rec}) that (\ref{Xrows}) holds for
$i+j\le d$. \ We may thus assume that $j= d+1-i$.
Consider first the case when $n+f+g<d$. \ In the subcase when $i \le d$, it follows from the presence of old moments in 
$B[n+f+g,d+1]$ that
\begin{eqnarray*} 
\langle X^{i}Y^{d+1-i},X^{n+f}Y^{g}\rangle &=&
\beta_{i+n+f,d+1-i+g},
\end{eqnarray*}
and in $M_{d}$ we have
\begin{eqnarray} \label{new}
\beta_{i+n+f,d+1-i+g} &=& \langle X^{n+f}Y^{g+1},X^{i}Y^{d-i}\rangle \nonumber \\
&=& \sum \limits_{r,s\ge 0, \ r+s\le n-1}^{} a_{rs} \langle X^{r+f}Y^{s+g+1},X^{i}Y^{d-i}\rangle \\
&=& \sum a_{rs} \langle X^{i}Y^{d-i}, X^{r+f}Y^{s+g+1}\rangle \quad \text{ \ (by symmetry in $M_{d}$)} \nonumber \\
&=& \sum a_{rs} \beta_{i+r+f,d-i+s+g+1} \nonumber \\
&=& \sum a_{rs} \langle X^{i}Y^{d+1-i}, X^{r+f}Y^{s+g}\rangle \nonumber \\
& & \text{ \ (by moment matrix structure in $B(d+1)$)}. \nonumber
\end{eqnarray}
For the subcase when $i=d+1$, we first note that $\langle X^{d+1}, X^{n+f}Y^{g}\rangle = \beta_{d+1+n+f,g} = \langle X^{d}, X^{n+f}Y^{g}\rangle 
= \langle X^{n+f}Y^{g}, X^{d}\rangle $, and we then proceed beginning as in (\ref{new}). \ 

We next consider the case $n+f+g=d$, 
and we seek to show that
\begin{equation}\label{XrowsHi}
\langle X^{i}Y^{d+1-i}, X^{n+f}Y^{g}\rangle=
\sum a_{rs} \langle X^{i}Y^{d+1-i}, X^{r+f}Y^{s+g}\rangle.
\end{equation}
We begin by showing that (\ref{XrowsHi}) holds if the
column $X^{i}Y^{d+1-i}$ is recursively determined
from (\ref{Xton_rec}), i.e., $i \ge n$. \ In this case,
we have $0\le i\le d+1-n$, so
\begin{eqnarray*}
\langle X^{i}Y^{d+1-i}, X^{n+f}Y^{g}\rangle &=& \sum a_{rs} \langle X^{r}Y^{s}X^{i-n}Y^{d+1-i},X^{n+f}Y^{g}\rangle \\
&=& \sum a_{rs} \langle  X^{n+f}Y^{g}, X^{r}Y^{s}X^{i-n}Y^{d+1-i}  \rangle \\
&=& \sum  a_{uv} \sum a_{rs} \langle  X^{u}Y^{v}X^{f}Y^{g}, X^{r}Y^{s}X^{i-n}Y^{d+1-i}  \rangle \\
&=& \sum  a_{uv} \sum a_{rs} \langle X^{r}Y^{s}X^{i-n}Y^{d+1-i}, X^{u+f}Y^{v+g}  \rangle \\
&=& \sum a_{uv}  \langle X^{i}Y^{d+1-i}, X^{u+f}Y^{v+g}\rangle.
\end{eqnarray*}
Thus $$ \langle X^{i}Y^{d+1-i}, X^{n+f}Y^{g}\rangle=
 \sum a_{uv}  \langle X^{i}Y^{d+1-i}, X^{u+f}Y^{v+g}\rangle,$$
which  is equivalent to (\ref{XrowsHi}). 

Returning to the proof of (\ref{XrowsHi}), we next assume
that column   $X^{i}Y^{d+1-i}$ is not recursively determined,
i.e., $d+1-m<i<n$. \ By the Hankel condition in $B(d+1)$, we have
\begin{eqnarray*}
\langle X^{i}Y^{d+1-i}, X^{n+f}Y^{g}\rangle &=& \langle X^{n}Y^{d+1-n}, X^{i+f}Y^{n-i+g}\rangle \\
&=&  \sum a_{rs} \langle X^{r}Y^{s+d+1-n},X^{i+f}Y^{n-i+g}\rangle \\
&=& \sum a_{rs} \beta_{r+i+f,s+d+1-i+g} \quad \text{(in $M_{d}$)} \\
&=& \sum a_{rs} \langle X^{i}Y^{d+1-i},X^{r+f}Y^{s+g}\rangle \\
&& \quad \text{(since $r+f+s+g<n+f+g=d$)}.
\end{eqnarray*}

Note that if (\ref{Xrows}) holds
for a collection of columns, then it holds for linear combinations
of those columns. \ Thus, using the preceding cases and (\ref{Y^m_rec}), we see that (\ref{Xrows})
holds, successively, for $X^{d+1-m}Y^{m},\ldots,Y^{d+1}$,
which completes the proof.
\end{proof}

The following result shows 
that the rows of
$\bpm
M_{d} &  B(d+1)  \epm$
 of the form $X^{f}Y^{m+g}$   $(f,g\ge 0, ~m+f+g\le d)$
are recursively determined from row $Y^{m}$.

\begin{lemma}\label{recYrows}   For $i,j\ge 0, i+j\le d+1$ and
for $f,g\ge 0,~ m+f+g\le d$,
\begin{equation}\label{Yrows}
     \langle X^{i}Y^{j},X^{f}Y^{m+g}\rangle =
\displaystyle \sum \limits_{u,v\ge 0, u+v\le m,v<m}
   b_{uv}\langle X^{i}Y^{j},  X^{u+f}Y^{v+g}\rangle.
\end{equation}
\end{lemma}
\begin{proof}
Since $M_{d}$ is real symmetric and recursively generated, its
rows are also recursively generated from (\ref{X}) and (\ref{Y}), so
(\ref{Yrows}) holds if $i+j\le d$. \ We may now assume $j=d+1-i$, and we
first consider the case $m+f+g<d$ and the subcase $i \le d$. \ Since $f+g+m<d$, using old moments we see that 
\begin{eqnarray*}
\langle X^{i}Y^{d+1-i},X^{f}Y^{m+g}\rangle &=& \beta_{i+f,d+1-i+g+m} \\
&=& \langle X^{f}Y^{m+g+1}, X^{i}Y^{d-i} \rangle \quad \text{(in $M_{d}$)} \\
&=& \sum b_{uv}  \langle X^{u+f}Y^{v+g+1}, X^{i}Y^{d-i} \rangle \\
&=& \beta_{i+u+f,d-i+v+g+1} \text{(in $M_{d}$)} \\
&=& \sum b_{uv} \langle X^{i}Y^{d+1-i}, X^{u+f}Y^{v+g} \rangle \\
&& \quad \text{(since $u+v+f+g<d$)}.
\end{eqnarray*}
The subcase when $i=d+1$ proceeds as above, but starting with $\langle X^{d+1},X^{f}Y^{m+g} \rangle = \beta_{d+1+f,m+g} =
\langle X^{d},X^{f+1}Y^{m+g} \rangle = \langle X^{f+1}Y^{m+g},X^{d} \rangle$. \ For the case $m+f+g=d$, 
we first consider the subcase when $i\ge n$,
so  $X^{i}Y^{d+1-i}$ is in the left recursive band.
We have 
\begin{eqnarray*}
\langle X^{i}Y^{d+1-i},X^{f}Y^{m+g}\rangle &=& \langle X^{n}X^{i-n}Y^{d+1-i},X^{f}Y^{m+g}\rangle \\
&=& \sum a_{rs}  \langle X^{r+i-n}Y^{s+d+1-i},X^{f}Y^{m+g}\rangle \\
&=& \sum a_{rs}  \sum b_{uv} \langle X^{r+i-n}Y^{s+d+1-i},X^{u+f}Y^{v+g}\rangle \\
&& \quad \text{(by row recursiveness in $M_{d}$)} \\
&=& \sum b_{uv} \langle X^{i}Y^{d+1-i},X^{n+f}Y^{v+g}\rangle.
\end{eqnarray*}

In the next subcase, we consider a column in the center band, of the form
$X^{d+1-i}Y^{i}$ with $d+1-n<i<m$. \ In this case,
(\ref{Yrows}) is equivalent to
\begin{equation}\label{newYrows}
     \langle X^{d+1-i}Y^{i},X^{f}Y^{m+g}\rangle =
\displaystyle \sum \limits_{u,v\ge 0, u+v\le m,v<m}
   b_{uv}\langle X^{d+1-i}Y^{i},  X^{u+f}Y^{v+g}\rangle.
\end{equation}
Note that the component $\langle X^{d+1-i}Y^{i},X^{f}Y^{m+g}\rangle$
lies on a cross-diagonal that reaches column $X^{d+1-m}Y^{m}$, so since $B(d+1)$ is well-defined, we have
\begin{eqnarray}\label{ymred}
\langle X^{d+1-i}Y^{i},X^{f}Y^{m+g}\rangle &=& \langle X^{d+1-m}Y^{m},X^{f+m-i}Y^{g+i}\rangle \nonumber \\
&=& \sum \limits_{u,v\ge 0, u+v\le m,v<m} b_{uv}\langle X^{u+d+1-m}Y^{v},  X^{f+m-i}Y^{g+i}\rangle.
\end{eqnarray}

 For the subcase when $u+v<m$, in $M_{d}$ we have
\begin{eqnarray*}
\langle X^{u+d+1-m}Y^{v}, X^{f+m-i}Y^{g+i}\rangle &=& \beta_{u+d+1+f-i,v+g+i} \\
&=& \langle X^{d+1-i}Y^{i},X^{u+f}Y^{v+g}\rangle \quad \text{(since $u+f+v+g\le d-1$)}.
\end{eqnarray*}

For the subcase when $u+v=m$, there are three further subcases
in showing that
\begin{equation}\label{mrowsred}
 \langle X^{d+1-i}Y^{i},X^{u+f}Y^{v+g}\rangle
= \langle X^{u+d+1-m}Y^{v},  X^{f+m-i}Y^{g+i}\rangle.
\end{equation}
For $v=i$, (\ref{mrowsred}) is clear.
For $v<i$, the Hankel property in $B[d,d+1]$ implies
\begin{eqnarray*}
\langle X^{d+1+u-m}Y^{v},X^{m+f-i}Y^{g+1} \rangle &=& \langle X^{d+1+u-m-(i-v)}Y^{v+(i-v)},X^{m+f-i+(i-v)}Y^{g+i-(i-v)} \rangle \\ 
&=& \langle X^{d+1-i}Y^{i},X^{u+f}Y^{g+v} \rangle .
\end{eqnarray*}
For $v>i$ we have, similarly,
\begin{eqnarray*}
\langle X^{d+1-i}Y^{i},X^{u+f}Y^{v+g}\rangle &=& \langle X^{d+1-i-(v-i)}Y^{i+v-i},X^{u+f+v-i}Y^{v+g-(v-i)}\rangle \\
&=& \langle X^{d+1+u-m}Y^{v},X^{m+f-i}Y^{g+i}\rangle.
\end{eqnarray*}
Since (\ref{Yrows}) holds in $M_{d}$ and in all columns of the left
and center bands, it now follows, using (\ref{Y^m_rec}) successively, that it holds
for columns in the right recursive band, which completes the proof.
\end{proof}

We are now prepared to prove that $Ran~B(d+1)\subseteq Ran~M_{d}$.
It follows immediately from (\ref{Xton_rec}) that each column in the left recursive band
of $B(d+1)$ belongs to $Ran~M_{d}$. \ In view of (\ref{Y^m_rec}), to  establish range inclusion, it suffices to show that each central-band column of $B(d+1)$ belongs to $Ran~M_{d}$. \ 
Let $\mathcal{S}$ denote the set of recursively determined
columns of $M_{d}$, i.e.,
$$
\mathcal{S} = \{X^{n},~X^{n+1},~X^{n}Y,\ldots, X^{d},\ldots,
X^{n}Y^{d-n},\ldots,Y^{m},~ XY^{m},~Y^{m+1},\dots,X^{d-m}Y^{m},
\ldots, Y^{d}\}.
$$ 
Let $\mathcal{B}$ denote the basis for
$Col~M_{d}$ (the column space of $M_{d}$) consisting of those
columns of $M_{d}$ which do not belong to $\mathcal{S}$.
Let $M_{\mathcal{B}}$ denote the compression of $M_{d}$ to the
rows and columns indexed by $\mathcal{B}$. \ Since $M_{d} \succeq 0$, we also have $M_{\mathcal{B}} \succeq 0$.
 Let $X^{i}Y^{d+1-i}$ ($d+1-m<i<n$)
denote a central-band column of $B(d+1)$, and let
$v_{i} \equiv [X^{i}Y^{d+1-i}]_{\mathcal{B}}$ denote the
compression of $X^{i}Y^{d+1-i}$ to the rows of $\mathcal{B}$.
There exists a unique vector of coefficients
$(c_{ab}^{(i)})_{X^{a}Y^{b}\in \mathcal{B}}$ such that
$$v_{i} = \displaystyle \sum \limits_{X^{a}Y^{b}\in \mathcal{B}}^{}
        c_{ab}^{(i)}[X^{a}Y^{b}]_{\mathcal{B}},$$
i.e., for each $X^{u}Y^{v}\in \mathcal{B}$,
\begin{equation}\label{range_formula}
    \langle X^{i}Y^{d+1-i}, X^{u}Y^{v}\rangle
=   \displaystyle \sum \limits_{X^{a}Y^{b}\in \mathcal{B}}^{} c_{ab}^{(i)}
\langle X^{a}Y^{b},X^{u}Y^{v}\rangle.
\end{equation}
To complete the proof that $Ran~B(d+1)\subseteq Ran~M_{d}$, it suffices
to prove that    $X^{i}Y^{d+1-i}
 = \displaystyle \sum \limits_{X^{a}Y^{b}\in \mathcal{B}}^{}
        c_{ab}^{(i)}X^{a}Y^{b},$
which, in view of (\ref{range_formula}), follows from the next result.
\begin{lemma}\label{range}
For each $X^{c}Y^{e}\in \mathcal{S}$,
\begin{equation}\label{range_equation}
    \langle X^{i}Y^{d+1-i}, X^{c}Y^{e}\rangle
=   \displaystyle \sum \limits_{X^{a}Y^{b}\in \mathcal{B}}^{} c_{ab}^{(i)}
\langle X^{a}Y^{b},X^{c}Y^{e}\rangle.
\end{equation}
\end{lemma}
\begin{proof}
We may assume without loss of generality that $n\le m$, so the
elements of $\mathcal{S}$ may be arranged in degree-lexicographic
order as $X^{n}, \cdots, Y^{m}, \cdots, X^{d}, \cdots, Y^{d}$. \
We will prove (\ref{range_equation}) by induction on the position number
of  row $X^{c}Y^{e}\in \mathcal{S}$ within the degree-lexicographic ordering.
For row $X^{n}$ ($c=n$, $e=0$), Lemma \ref{recXrows} implies that
\begin{equation}\label{basecase}
 \langle X^{i}Y^{d+1-i}, X^{n}\rangle=
\displaystyle
 \sum \limits_{r,s\ge 0, r+s\le n-1} 
a_{rs} \langle X^{i}Y^{d+1-i}, X^{r}Y^{s}\rangle.
\end{equation}
Since $r+s<n$, $X^{r}Y^{s}\in \mathcal{B}$, so the sum in
(\ref{basecase}) may be expressed as
$$     \sum a_{rs}
   \displaystyle \sum \limits_{X^{a}Y^{b}\in \mathcal{B}}^{} c_{ab}^{(i)}
\langle X^{a}Y^{b}, X^{r}Y^{s}\rangle = 
   \displaystyle \sum \limits_{X^{a}Y^{b}\in \mathcal{B}}^{} c_{ab}^{(i)}
\langle X^{a}Y^{b}, X^{n}\rangle$$
(using Lemma \ref{recXrows} again). \ 
 Assume now that   (\ref{range_equation})
holds for all rows $X^{C}Y^{E}\in \mathcal{S}$ with 
order position up to $k-1$, and consider
$X^{c}Y^{e}\in \mathcal{S}$ with position $k$. \ Either $c\ge n$ or
$e\ge m$; we present the argument for the case $e\ge m$ (the other case
is simpler).
We have $e= m+g$ for some $g\ge 0$. \  
 From Lemma \ref{recYrows}, we have
$$ \langle X^{i}Y^{d+1-i}, X^{c}Y^{m+g}\rangle=
\displaystyle 
\sum \limits_{u,v\ge 0, u+v\le m, v<m}
b_{uv} \langle X^{i}Y^{d+1-i}, X^{c+u}Y^{g+v}\rangle. $$
Now $X^{c+u}Y^{g+v}$ is either a basis vector, or, since $v<m$,
it precedes $X^{c}Y^{m+g}$ in the ordering of $\mathcal{S}$.
Thus, by definition (for the basis rows)
and by induction (for the non-basis rows), the preceding sum is
equal to
$$=     \sum b_{uv}
   \displaystyle \sum \limits_{X^{a}Y^{b}\in \mathcal{B}}^{} c_{ab}^{(i)}
\langle X^{a}Y^{b}, X^{c+u}Y^{g+v}\rangle = 
   \displaystyle \sum \limits_{X^{a}Y^{b}\in \mathcal{B}}^{} c_{ab}^{(i)}
\langle X^{a}Y^{b},
 \sum b_{uv} X^{c+u}Y^{g+v}\rangle$$
 $$=
   \displaystyle \sum \limits_{X^{a}Y^{b}\in \mathcal{B}}^{} c_{ab}^{(i)}
\langle X^{a}Y^{b},
  X^{c}Y^{e}\rangle$$
(by another application of Lemma  \ref{recYrows}).
\end{proof}
The proof of Theorem \ref{main} is now complete. \   


\section{Proof of Theorem \ref{RDnew}} \label{PROOF2}

For the proof of Theorem \ref{RDnew}, we require a preliminary result concerning a general moment matrix.
\begin{lemma}\label{longcolumns}
Suppose $M_{d+1}$ satisfies $Ran~B(d+1)\subseteq
Ran~M_{d}$. \ If $p\in \mathcal{P}_{d}$ and
$p(X,Y) = 0$ in $Col~M_{d}$, then $p(X,Y) = 0$
in $Col~M_{d+1}$.
\end{lemma} 
\begin{proof}
Since $M_{d}$ is real symmetric, we have $p(X,Y) = 0$
in the row space of $M_{d}$, and we first show that
$p(X,Y)=0$ holds in the row space of
$\bpm
M_{d} & B(d+1)   
\epm.$ 
Let $\rho := deg~p$ and
suppose $p(x,y) \equiv \sum_{r,s\ge 0, r+s\le \rho} a_{rs}x^{r}y^{s}$.
Then for $i,j\ge 0$ with $i+j\le d$, we have
\begin{equation}\label{pM}
\sum_{r,s} \alpha_{rs} \langle X^{i}Y^{j}, X^{r}Y^{s} \rangle = 0.
\end{equation}
Consider a column  of degree $d+1$, $X^{u}Y^{d+1-u}$ ($0\le u\le d+1$).
We seek to show that
\begin{equation}\label{pB}
\sum_{r,s} \alpha_{rs} \langle X^{u}Y^{d+1-u}, X^{r}Y^{s} \rangle = 0.
\end{equation}
By the range inclusion, we have a dependence relation in $Col~
 \bpm
M_{d} & B(d+1)   
\epm$ of the form
\begin{equation}\label{nextdegree}
X^{u}Y^{d+1-u} = \sum_{a,b\ge 0, a+b\le d} c_{ab}^{(u)} X^{a}Y^{b}.
\end{equation}
Thus,
\begin{eqnarray*}
\sum_{r,s} \alpha_{rs} \langle X^{u}Y^{d+1-u}, X^{r}Y^{s} \rangle &=&
\sum_{r,s} \alpha_{rs}  \sum_{a,b\ge 0, a+b\le d} c_{ab}^{(u)}
\langle X^{a}Y^{b}, X^{r}Y^{s} \rangle \\
&=& \sum  c_{ab}^{(u)}     \sum \alpha_{rs}   
\langle X^{a}Y^{b}, X^{r}Y^{s} \rangle = 0 \quad (\text{by (\ref{pM})}).
\end{eqnarray*}
Now, $p(X,Y)=0$ in 
the row space of
$
 \bpm
M_{d} & B(d+1)   
\epm$,
so  
$p(X,Y)=0$ in $Col~
\bpm
M_{d} \\
 B(d+1)^{T}   
\epm.$
\end{proof}

\begin{proof} [Proof of Theorem \ref{RDnew}]
\ It follows from the proof of Theorem \ref{gridthm} that $M_{d}$ admits
a unique extension $M_{d+1}$ which satisfies $Ran~B(d+1)\subseteq
Ran~M_{d}$ and such that (\ref{Xton_rec})-(\ref{Y^m_rec}) hold in $Col~M_{d+1}$.
It remains only to prove that $M_{d+1}$ is recursively
generated. \ 
Since $M_{d}$ is recursively generated, it suffices to consider
a dependence relation in $Col~M_{d+1}$ of degree $d$,
say
\begin{equation}\label{firstrelation}
X^{i}Y^{d-i} = \sum_{g,h\ge0,g+h\le d-1} c_{gh}X^{g}Y^{h}
\end{equation}
(where $0\le i\le d$),
and to show that
\begin{equation}\label{timesX}
X^{i+1}Y^{d-i} = \sum_{g,h\ge0,g+h\le d-1} c_{gh}X^{g+1}Y^{h}
\end{equation}
and
\begin{equation}\label{timesY}
X^{i}Y^{d-i+1} = \sum_{g,h\ge0,g+h\le d-1} c_{gh}X^{g}Y^{h+1}.
\end{equation}
Suppose first that $i\ge n$, so that $X^{i}Y^{d-i}$ lies in the
left band. \ Then from (\ref{Xrec}) we also have
\begin{equation}\label{secondrelation}
X^{i}Y^{d-i} = \sum_{r+s\le n-1} a_{rs} X^{i-n+r}Y^{s+d-i}.
\end{equation}
Thus, in $M_{d}$ we have the column relation of degree
at most $d-1$,
$$\sum_{g+h\le d-1} c_{gh}X^{g}Y^{h} 
= \sum_{r+s\le n-1} a_{rs} X^{i-n+r}Y^{s+d-i}.$$
Since $M_{d}$ is recursively generated, it follows that in
$Col~M_{d}$ we also have
$$\sum_{g+h\le d-1} c_{gh}X^{g+1}Y^{h} 
= \sum_{r+s\le n-1} a_{rs} X^{i-n+r+1}Y^{s+d-i}.$$
Lemma \ref{longcolumns} implies that the
last equation also holds  in $Col~M_{d+1}$, where, from (\ref{Y^m_rec}),
the right-hand sum represents $X^{i+1}Y^{d-i}$; this 
establishes (\ref{timesX}). \ We omit the proof of  (\ref{timesY}),
which is similar. \  The case when $d-i\ge m$, so that
 $X^{i}Y^{d-i+1}$ is in the right band, is handled in an entirely 
analogous fashion, so we also omit the proof of this case.

We next consider the case when $d-m<i<n$, so that column
$X^{i}Y^{d-i}$ in (\ref{firstrelation}) is in the central band.
To establish (\ref{timesX}), it suffices to verify that
\begin{equation}\label{Xcomponent}
\langle X^{i+1}Y^{d-i}, X^{k}Y^{j} \rangle 
= \sum_{g,h\ge0,g+h\le d-1} c_{gh} \langle X^{g+1}Y^{h},X^{k}Y^{j}\rangle \quad (k,j\ge 0,~k+j\le d+1).
\end{equation}
The case when $k+j<d$ is easy, using (\ref{firstrelation}) and
the old moments in  block $B[k+j,d+1]$. \ We consider next the
case $k+j=d$ and the subcase when $k\ge n$. \
In this subcase, $\langle X^{i+1}Y^{d-i}, X^{k}Y^{d-k} \rangle$
belongs to a cross-diagonal of $B[d,d+1]$ 
that intersects column $X^{n}Y^{d+1-n}$,
so
from the definition
of $B[d,d+1]$ in the proof of Theorem \ref{main}, we have
\begin{eqnarray} \label{recformula}
\langle X^{i+1}Y^{d-i}, X^{k}Y^{d-k} \rangle
&:=& \langle X^{n}Y^{d+1-n}, X^{k-(n-i-1)}Y^{d-k+n-i-1} \rangle \nonumber \\
&=& \sum a_{rs} \langle X^{r}Y^{s+d+1-n}, X^{k-(n-i-1)}Y^{d-k+n-i-1} \rangle.
\end{eqnarray}
Now, we have
\begin{eqnarray*}
\sum_{g,h\ge 0,g+h\le d-1} c_{gh} \langle X^{g+1}Y^{h},X^{k}Y^{d-k}\rangle &=& \sum c_{gh} \langle X^{k}Y^{d-k},X^{g+1}Y^{h}\rangle \\
&=& \sum c_{gh} \sum_{rs} a_{rs} \langle X^{r+k-n}Y^{s+d-k},X^{g+1}Y^{h}\rangle \\
&=& \sum a_{rs} \sum  c_{gh}  \langle X^{g+1}Y^{h},X^{r+k-n}Y^{s+d-k}\rangle \\
&=& \sum a_{rs} \sum  c_{gh}  \langle X^{g}Y^{h},X^{r+k-n+1}Y^{s+d-k}\rangle \quad (\text{in} \; M_{d}) \\
&=& \sum a_{rs}  \langle X^{i}Y^{d-i},X^{r+k-n}Y^{s+d-k}\rangle \\
&=& \sum a_{rs}  \langle   X^{r+k-n}Y^{s+d-k},  X^{i}Y^{d-i}\rangle \\
&=& \sum a_{rs}  \langle   X^{r}Y^{s+d-k+(k-n+1)},  X^{i+(k-n+1)}Y^{d-i-(k-n+1)}\rangle.
\end{eqnarray*}
This last expression agrees with (\ref{recformula}), so (\ref{timesX})
is established for this subcase. \ The proof of this subcase for
(\ref{timesY}) is very similar, so we omit the details. \ In the subcase
when $k< n$, then $d-k\ge m$, and we see that
$\langle X^{i+1}Y^{d-i}, X^{k}Y^{d-k} \rangle$
belongs to a cross-diagonal of $B[d,d+1]$ 
that intersects column $X^{d+1-m}Y^{m}$. \ Since $deg~q<m$,
the proof of this subcase is entirely analogous to that above, 
but using (\ref{Y^m_rec}) for the definition of $X^{d+1-m}Y^{m}$. 

Finally, we consider the case $k+j=d+1$. \ As above, we will treat the
subcase of (\ref{timesX}) when $k\ge n$ in detail and omit the proofs of
the other subcases of (\ref{timesX}) and (\ref{timesY}), which are similar.
Since $k\ge n$, then, as above, we have
\begin{eqnarray} \label{recformulaC}
\langle X^{i+1}Y^{d-i}, X^{k}Y^{d+1-k} \rangle &:=& \langle X^{n}Y^{d+1-n}, X^{k-(n-i-1)}Y^{d+1-k+n-i-1} \rangle \nonumber \\ 
&=& \sum a_{rs} \langle X^{r}Y^{s+d+1-n}, X^{k-(n-i-1)}Y^{d-k+n-i} \rangle.
\end{eqnarray}
Now, 

\begin{eqnarray*}
\sum_{g,h\ge 0,g+h\le d-1} c_{gh} \langle X^{g+1}Y^{h},X^{k}Y^{d+1-k}\rangle \\
&=& \sum c_{gh} \langle X^{k}Y^{d+1-k},X^{g+1}Y^{h}\rangle \\
&& (\text{since} \bpm M_{d} & B(d+1) \epm \text{is the transpose of} \\
&&  \bpm
M_{d} \\
 B(d+1)^{T}   
\epm) \\
&=& \sum  c_{gh} \sum_{rs} a_{rs}  \langle X^{r+k-n}Y^{s+d+1-k},X^{g+1}Y^{h}\rangle \\
&=& \sum a_{rs} \sum  c_{gh}  \langle X^{g+1}Y^{h},X^{r+k-n}Y^{s+d+1-k}\rangle.
\end{eqnarray*}
Since the row degrees of the terms in the last sum are at most $d$, by the
previous cases (for $j+k<d$ and $j+k=d$), the last double sum may be
expressed as 
$$ \sum a_{rs}  \langle X^{i+1}Y^{d-i},X^{r+k-n}Y^{s+d+1-k}\rangle$$
relative to $\bpm
M_{d} & B(d+1)   
\epm$. \ Since $M_{d+1}$ is real symmetric, the latter sum may be
expressed as
\begin{eqnarray*}
\sum a_{rs} \langle X^{r+k-n}Y^{s+d+1-k},
X^{i+1}Y^{d-i}\rangle \\
&=& \sum a_{rs}   \langle X^{r}Y^{s+d+1-k+(k-n)},
X^{i+1+(k-n)}Y^{d-i-(k-n)}\rangle \\
&=& \sum a_{rs}   \langle X^{r}Y^{s+d+1-n)},
X^{i+1+k-n}Y^{d-i-k+n}\rangle,
\end{eqnarray*}
and this agrees with (\ref{recformulaC}).
The proof is now complete.
\end{proof}


\bigskip

\noindent
Ra\' ul E. Curto

\noindent
Department of Mathematics

\noindent
The University of Iowa

\noindent
Iowa City, Iowa 52246, USA

\noindent Email: raul-curto@uiowa.edu

\medskip

\noindent
Lawrence Fialkow

\noindent
Departments of Computer Science and Mathematics

\noindent
State University of New York

\noindent
New Paltz, New York 12561

\noindent Email: fialkowl@newpaltz.edu

\end{document}